\documentclass[a4paper,10pt]{article}
\usepackage[dvipdfmx]{graphicx}
\usepackage{float}
\usepackage{latexsym,amssymb,amsmath,mathrsfs}
\usepackage{amsthm}
\usepackage{mathrsfs}
\usepackage[english]{babel}
\usepackage{comment}
\newtheorem{theorem}{Theorem}[section]
\newtheorem{remark}{Remark}[section]
\newtheorem{lemma}{Lemma}[section]
\newtheorem*{lemma*}{Lemma}

\newtheorem*{example*}{Example}

\newtheorem*{proposition*}{Prop}

\newtheorem*{question*}{Quest}

\newtheorem*{assume*}{Assume}

\newtheorem*{claim*}{Claim}
\newtheorem{corollary}{Corollary}[section]
\DeclareMathOperator{\Ric}{Ric}

\newcommand{\R}{\mathbb{R}}

\newcommand{\td}{\mathrm{d}}

\newcommand{\lra}{\longrightarrow}

\newcommand{\supp}{\mathop{\rm supp}\nolimits}

\newcommand{\ez}{\varepsilon}
\newcommand{\Lap}{\Delta}
\newcommand{\bs}{\ensuremath{\mathbf{s}}}
\newcommand{\e}{\mathrm{e}}

\renewcommand{\d}{\mathrm{d}}

\newcommand{\eps}{{\varepsilon}}

\def\Cut{\mathop{\mathrm{Cut}}\nolimits}
\title{Analysis of harmonic functions under lower bounds of $N$-weighted Ricci curvature with $\eps$-range}
\author{Yasuaki Fujitani\thanks{Department of Mathematics, Osaka University, Osaka 560-0043, Japan (\texttt{u197830k@ecs.osaka-u.ac.jp})}}
\begin{document}
\maketitle
\begin{abstract}
    The behavior of harmonic functions on Riemannian manifolds under lower bounds of the Ricci curvature has been studied from both analytic and geometric viewpoints.
    For example, some Liouville type theorems are obtained under lower bounds of the Ricci curvature. Recently, those results are generalized under lower bounds of the $N$-weighted Ricci curvature $\Ric_\psi^N$ with $N \in [n,\infty]$. 
    In this paper, we present a Liouville type theorem for harmonic functions of sublinear growth and a gradient estimate of harmonic functions on weighted Riemannian manifolds under weaker lower bounds of $\Ric_\psi^N$ with $N  < 0$.
    We also prove an $L^p$-Liouville theorem under lower bounds of $\Ric_\psi^N$ with $N \in [n,\infty]$ in a way different from that of \cite{jy-1}.
    Our results are obtained as a consequence of an argument under lower bounds of $\Ric_\psi^N$ with $\eps$-range, which is a unification of constant and variable curvature bounds.
    Among various methods considered for the analysis of harmonic functions, this paper focuses on methods using the Moser's iteration procedure.
\end{abstract}
\tableofcontents
\section{Introduction}

The analysis of harmonic functions is closely related to comparison geometry concerning lower bounds of the Ricci curvature.
As a pioneering work, Yau showed in 1975 that harmonic functions on complete Riemannian manifolds with $\Ric\geq 0$ are limited to constants \cite{yau-1975}.
Since then, such a Liouville-type theorem has been intensively studied and there are recently a number of generalizations to weighted Riemannian manifolds.

This paper is concerned with the $N$-weighted Ricci curvature $\Ric_\psi^N$, which is a generalization of the Ricci curvature for weighted Riemannian manifolds. Here,  $\psi$ is a weight function and $N$ is a real parameter called the effective dimension.
Many comparison geometric results similar to those for Riemannian manifolds with Ricci curvature bounded from below by $K$ and dimension bounded from above by $N$ also hold if we assume the condition $\Ric_\psi^N\geq K$ for $K\in\mathbb{R}$. 
Especially the case of $N\in [n,\infty]$ is now classical and well investigated.
For example, under lower bounds of $\Ric_\psi^{\infty}$, 
a Liouville type theorem for harmonic functions of sublinear growth and 
a gradient estimate of harmonic functions were obtained in \cite{munteau_wang-1,soliton}. These results were also new in that they were proved by combining the Moser's iteration argument and the Bochner inequality. 
In \cite{jy-1}, Wu obtained some Liouville type theorems for $L^p$-functions, which is a generalization of \cite{li-schoen}, by using the Moser's iteration argument.

Recently, the $N$-weighted Ricci curvature with $N \in (-\infty,1]$ is getting more attention. 
For this range, some Poincar\'{e} inequalities \cite{milman} (see also \cite{mai} for its rigidity), Beckner inequality \cite{gentil} and the curvature-dimension condition \cite{ohta-negative,needle} were studied.
Considering the range $N\in(-\infty,1]$ is also meaningful in the analysis of porous medium equations, for which we refer to \cite{ohta-takatsu,yasuaki2}, for instance.

In this paper, we prove Liouville type theorems under lower bounds of $\Ric_\psi^N$ for some $N\in (-\infty,1]$.
Although it is known that the Liouville type theorem under lower bounds of $\Ric_\psi^N$ for $N\in[n,\infty)$ can also be obtained by using the Li-Yau gradient estimate as in \cite{wu}, this method cannot be directly applied especially in the region of $N \leq 1$. In this paper, we take an alternative approach which uses the Moser's iteration argument. 

We need the Sobolev inequality and some volume growth estimates to conduct the Moser's iteration procedure.
Since some comparison theorems (such as the Bishop-Gromov volume comparison theorem and the Laplacian comparison theorem) under constant curvature bounds $\Ric_\psi^N\geq K$ hold only for $N\in[n,\infty)$ and fail for $N\in (-\infty,1]\cup \{\infty\}$, 
we face essential difficulties when trying to conduct the Moser's iteration procedure under constant lower bounds of $\Ric_\psi^N$ with $N \in(-\infty,1]\cup\{\infty\}$.
Therefore, we consider variable curvature bounds called the $\eps$-range and conduct the procedure under lower bounds of $\Ric_\psi^N$ with $\eps$-range. 

We briefly review the notion of $\eps$-range, which is a unification of constant and variable curvature bounds by choosing appropriate $\eps$.
Wylie-Yeroshkin \cite{yoroshikin} first introduced a variable curvature bound 
\begin{equation*}
    \Ric_\psi^1 \geq K\e^{-\frac{4}{n-1}\psi}g
\end{equation*}
associated with the weight function $\psi$, and established several comparison theorems including the Bishop-Gromov volume comparison theorem. They were then generalized to 
\begin{equation*}
    \Ric_\psi^N\geq K\e^{\frac{4}{N-n}\psi}g
\end{equation*}
with $N\in (-\infty, 1)$ by Kuwae-Li \cite{kuwae_li}. In \cite{lu}, Lu-Minguzzi-Ohta gave a further generalization of the form 
\begin{equation}\label{eps-range-intro}
    \Ric_\psi^N\geq K\e^{\frac{4(\eps-1)}{n-1}\psi}g
\end{equation}
for an additional parameter $\eps$ in an appropriate range called the $\eps$-range which depends on $N$ (we also refer to \cite{lu2} for a preceding work on singularity theorems in Lorentz-Finsler geometry).

In this paper, using mean value inequalities obtained by the Moser's iteration procedure under lower bounds of $\Ric_\psi^N$ with $\eps$-range, we investigate the behavior of harmonic functions. 
One of our main contributions is to present different proofs of the $L^p$-Liouville theorem under lower bounds of $\Ric^{\infty}_{\psi}$ in \cite{jy-1}
(see also Remark \ref{remark-elliptic-harnack}). 
We also contribute to a generalization of the Liouville theorem for harmonic functions of sublinear growth (Theorem \ref{liouville-negative}), which is obtained in \cite{munteau_wang-1} under $\Ric_\psi^{\infty}\geq 0$.
We can obtain this generalization as a consequence of arguments under lower bounds of $\Ric_\psi^N$ with $\eps$-range because the variable curvature bound with $\eps$-range degenerates into the constant curvature bound when we consider the non-negative curvature.
In addition, we show a gradient estimate for harmonic functions under lower bounds of $\Ric_\psi^N$ with $\eps$-range (Theorem \ref{grad-est-eps-range}), which is similar to the gradient estimate in \cite{soliton} under lower bounds of $\Ric_\psi^\infty$.

The organization of this paper is as follows. 
In Section 2, we briefly review the weighted Ricci curvature with $\eps$-range, the elliptic Harnack inequality obtained by the Moser's iteration argument, various Liouville theorems and the local Poincar\'{e} inequality. 
In Section 3, we show a mean value inequality under lower bounds of $\Ric_\psi^N$ with $\eps$-range.
In Section 4, we prove $L^p$-Liouville type theorems under lower bounds of $\Ric_\psi^N$ with $N\in[n,\infty]$ in two different ways.
In Section 5, we show a Liouville type theorem for harmonic functions of sublinear growth under $\Ric_\psi^N\geq 0$ for $N\in (-\infty,0)\cup [n,\infty]$ with bounded weight functions. 
In Section 6, we present a gradient estimate of harmonic functions under lower bounds of $\Ric_\psi^N$ for $N\in (-\infty,0)\cup [n,\infty]$ with  $\eps$-range using the elliptic Harnack inequality.

\section{Preliminaries}
\subsection{$\eps$-range}
Let $(M,g,\mu)$ be an $n$-dimensional weighted Riemannian manifold.
In this paper, we assume that the diameter of $M$ is not bounded.
We set $\mu = \e^{-\psi}v_g$ where $v_g$ is the Riemannian volume measure and $\psi$ is a $C^{\infty}$ function on $M$.
For $N\in(-\infty,1]\cup[n,+\infty]$, the \emph{$N$-weighted Ricci curvature} is defined as follows:
\begin{align*}
\Ric_\psi^N:=\Ric_g+{\rm \nabla^2}\psi-\frac{\d\psi\otimes \d\psi}{N-n},
\end{align*}
where when $N=+\infty$,
the last term is interpreted as the limit $0$
and when $N=n$,
we only consider a constant function $\psi$, and
set $\Ric_\psi^n:=\Ric_g$.
 
In \cite{lu}, they introduced the notion of \textit{$\eps$-range}:
\begin{equation}\label{epsilin-range}
\eps = 0 \mbox{ for } N = 1, \quad |\eps| < \sqrt{\frac{N-1}{N-n}} \mbox{ for } N\neq 1,n, \quad \eps\in \mathbb{R}\mbox{ for } N = n.
\end{equation}
In this $\eps$-range, for $K\in\mathbb{R}$, they considered the condition
\[ \Ric_\psi^N(v)\ge K \e^{\frac{4(\ez-1)}{n-1}\psi(x)} g(v,v),\quad v\in T_xM.\]
We also define a constant $c$ associated with $\eps$ as
\begin{equation}\label{c_num}
c = \frac{1}{n-1}\left(1- \eps^2\frac{N-n}{N-1}\right) > 0
\end{equation}
for $N\neq 1$ and $c = (n-1)^{-1}$ for $N = 1$.
We define the comparison function $\bs_{\kappa}$ as
\begin{equation}\label{eq:bs}
\bs_{\kappa}(t) := \begin{cases}
\frac{1}{\sqrt{\kappa}} \sin(\sqrt{\kappa}t) & \kappa>0, \\
t & \kappa=0, \\
\frac{1}{\sqrt{-\kappa}} \sinh(\sqrt{-\kappa}t) & \kappa<0.
\end{cases}
\end{equation}
We denote $B_x(r)=\{ y \in M \,|\, d(x,y)<r \}$, $B_x(r,R) = \{y\in M\  | \ r < d(x,y) < R\}$, $V_x(r) = \mu(B_x(r))$ and $V_x(r,R) = \mu(B_x(r,R))$. 
We set 
\begin{equation*}
    K_\eps(x) = \max\left\{ 0, \sup_{\ v\in U_xM}\left(-\e^{\frac{-4(\eps-1)}{n-1}\psi(x)} \Ric_{\psi}^N(v) \right)\right\}
\end{equation*}
and
\begin{equation*}
    K_\eps(q,R) =  \sup_{x \in B_q(R)}K_\eps(x)
\end{equation*}
for $q\in M$ and $R > 0$, where $U_x M$ denotes the set of unit vectors in $T_xM$. We note that we have $K_{\eps}(x)\geq 0$ and $K_{\eps}(q,R)\geq 0$.
The \emph{weighted Laplacian} $\Delta_\psi$ on $(M,g,\mu)$ is defined as 
\begin{equation*}
    \Delta_\psi = \Delta - \langle \nabla\psi,\nabla \cdot  \rangle.
\end{equation*}
A smooth function $u$ is called a \emph{$\psi$-harmonic function} when $\Delta_\psi u = 0$ and a $\psi$-subharmonic function when $\Delta_\psi u\geq 0$.

We briefly review the argument in \cite{lu} (where, more generally, Finsler manifolds equipped with measures were considered).
For a unit tangent vector $v \in T_xM $,
let $\eta:[0,l) \lra \R$ be the geodesic with $\dot{\eta}(0)=v$.
We take an orthonormal basis $\{e_i\}_{i=1}^n$ of $T_xM$ with $e_n=v$
and consider the Jacobi fields
\[ E_i(t):=(\td\exp_x)_{tv}(te_i), \quad i=1,2,\ldots,n-1, \]
along $\eta$.
We define the $(n-1) \times (n-1)$ matrix $A(t)=(a_{ij}(t))$ by
\[ a_{ij}(t):=g \big( E_i(t),E_j(t) \big). \]
We set
\[ h_0(t):=(\det A(t))^{1/2(n-1)}, \qquad
h(t) := \e^{-c\psi(\eta(t))}\big(\! \det A(t) \big)^{c/2}, \qquad
h_1(\tau):=h \big( \varphi_\eta^{-1}(\tau) \big) \]
for $t \in [0,l)$ and $\tau \in [0,\varphi_{\eta}(l))$, where
\begin{equation}\label{varphi-def}
\varphi_{\eta}(t) :=\int_0^t \e^{\frac{2(\ez -1)}{n-1}\psi(\eta(s))} \,\td s.
\end{equation}
By the definition, we have the following relationship:
\[ (\e^{-\psi(\eta)} h_0^{n-1})(t) =h(t)^{1/c} =h_1\big( \varphi_{\eta}(t) \big)^{1/c}. \]
According to the argument in \cite[Theorem 3.6]{lu}, the condition $\Ric_\psi^N(v)\ge K \e^{\frac{4(\ez-1)}{n-1}\psi(x)} g(v,v)$ implies that
\begin{equation}\label{non_increasing_property}
(\e^{-\psi(\eta)} h_0^{n-1})/\bs_{cK}(\varphi_{\eta})^{1/c} \mbox{ is non-increasing.}
\end{equation}
This plays a key role in proving the following Laplacian comparison theorem and the Bishop-Gromov volume comparison theorem.
 
\begin{theorem}(\cite[Theorem 3.9]{lu}, Laplacian comparison theorem)\label{laplacian_comparison_theorem}
    Let $(M, g, \mu)$ be an $n$-dimensional complete weighted Riemannian manifold and $N \in (-\infty,1] \cup [n,+\infty]$,
    $\ez \in \R$ in the $\ez$-range \eqref{epsilin-range}, $K \in \R$ and $b \ge a>0$.
    Assume that
    \[ \Ric_\psi^N(v)\ge K \e^{\frac{4(\ez-1)}{n-1}\psi(x)} g(v,v)\]
    holds for all $v \in T_xM \setminus 0$ and
    \[ a \le \e^{-\frac{2(\ez-1)}{n-1}\psi} \le b. \]
    Then, for any $z \in M$, the distance function $r(x):=d(z,x)$ satisfies
    \[ \Lap_{\psi} r(x) \le \frac{1}{c\rho} \frac{\bs'_{cK}(r(x)/b)}{\bs_{cK}(r(x)/b)} \]
    on $M \setminus (\{z\} \cup \Cut(z))$, where $\rho:=a$ if $\bs'_{cK}(r(x)/b) \ge 0$
    and $\rho:=b$ if $\bs'_{cK}(r(x)/b)<0$ and $\Cut(z)$ denotes the cut locus of $z$.
\end{theorem}

\begin{theorem}(\cite[Theorem 3.11]{lu}, Bishop-Gromov volume comparison theorem)\label{bishop_gromov}
    Under the same assumptions as Theorem \ref{laplacian_comparison_theorem},
    we have
    \[ \frac{V_x(R)}{V_x(r)}
    \le \frac{b}{a}
    \frac{\int_0^{\min\{R/a,\,\pi/\sqrt{cK}\}} \bs_{cK}(\tau)^{1/c} \,\d\tau}{\int_0^{r/b} \bs_{cK}(\tau)^{1/c} \,\d\tau} \]
    for all $x\in M$ and $0<r<R$, where $R \le b\pi/\sqrt{cK}$ when $K>0$
    and we set $\pi/\sqrt{cK}:=\infty$ for $K \le 0$.
\end{theorem}

Using this volume comparison theorem, we obtain the following Neumann-Poincar\'{e} inequality and local Sobolev inequality.

\begin{theorem}(\cite[Theorem 7]{yasuaki}, Neumann-Poincaré inequality)\label{poincare_thm}
    Let $(M,g,\mu)$ be an $n$-dimensional complete weighted Riemannian manifold and $N \in (-\infty,1] \cup [n,+\infty]$,
    $\ez \in \R$ in the $\ez$-range \eqref{epsilin-range}, $K > 0$ and $b \ge a>0$.
    Assume that
    \[ \Ric_\psi^N(v)\ge -K\e^{\frac{4(\ez-1)}{n-1}\psi(x)} g(v,v) \]
    holds for all $v \in T_xM \setminus 0$ and
    \begin{equation}
    a \le \e^{-\frac{2(\ez-1)}{n-1}\psi} \le b.
    \end{equation} 
    Then, we have
    \begin{equation}\label{poincare_inequality}
    \forall f \in C^{\infty}(M), \quad \int_{B_q(r)}|f - f_{B_q(r)}|^2 \ \d\mu \leq 2^{n+3}\left(\frac{2b}{a}\right)^{1/c}\exp\left(\sqrt{\frac{K}{c}}\frac{2r}{a}\right)r^2\int_{B_q(2r)}|\nabla f|^2 \ \d\mu
    \end{equation}
    for all $q\in M$ and $r > 0$, where
    \begin{equation*}
    f_{B_q(r)} := \frac{1}{V_q(r)}\int_{B_q(r)} f\  \d\mu.
    \end{equation*}
    \end{theorem}

\begin{theorem}(\cite[Theorem 8]{yasuaki}, Local Sobolev inequality)\label{sobolev_thm}
    Under the same assumptions as Theorem \ref{poincare_thm} with $n\geq 3$,
     there exist positive constants $D, E$ depending on $c,a,b, n$ such that 
    \begin{equation*}
    \left(V_q(r)^{-1}\int_{B_q(r)}|f|^{\frac{2(1 + c)}{1-c}}\d\mu\right)^{\frac{1-c}{1 + c}}\leq E\exp\left(D\left(1 + \sqrt{\frac{K}{c}}\right)\frac{r}{a}\right)r^2 V_q(r)^{-1}\int_{B_q(r)}(|\nabla f|^2 + r^{-2} f^2)\ \d\mu
    \end{equation*}
    for all $q\in M$, $r > 0$ and $f\in C_0^{\infty}(B_q(r))$.
 \end{theorem}

\subsection{Elliptic Harnack inequality} 
Using the Moser's iteration argument, we can obtain several mean value inequalities.
Combining these mean value inequalities, we can prove the elliptic Harnack inequality, which is one of the main applications of the Moser's iteration argument.
This subsection collects known results concerning these issues on unweighted Riemannian manifolds (i.e. $\mu = v_g$), which are generalized in Section 3.

In the following arguments in this paper, we fix a point $q\in M$.
For $p > 0$, we denote the $L^p$-norm on $B_q(R)$ by
\begin{equation*}
    \|f\|_{p,R} = \left( \int_{B_q(R)}|f|^p \ \d\mu \right)^{1/p}, \quad \|f\|_{\infty, R} = \sup_{B_q(R)}|f|.
\end{equation*}
Let $H^{1,2}(B_q(R))$ be the $(1,2)$-Sobolev space on $B_q(R)$ and $H_c^{1,2}(B_q(R))$ be the set of compact support elements in $H^{1,2}(B_q(R))$. 
We first present the following mean value inequality.
\begin{theorem}(\cite[Lemma 11.1]{p_li})\label{sub-lemma}
 Let $(M,g)$ be an $n$-dimensional complete Riemannian manifold and $u$ be a non-negative function in $H^{1,2}\left(B_q(R)\right)$ satisfying
$$
\Delta u \geq-f u,
$$
where $f$ is a non-negative function on $B_q(R)$.
Let $\nu_1=\frac{n}{2}$ when $n>2$, and $1<\nu_1<\infty$ be arbitrary when $n=2$. We assume that the $L^p$-norm of $f$ is finite for some $\nu_1<p \leq \infty$,
and that, for $\nu_2 > 0$ satisfying $\frac{1}{\nu_1}+\frac{1}{\nu_2}=1$, there exists a constant $C_s>0$ such that 
\begin{equation}\label{sobolev-constant-ineq}
\frac{1}{V_q(R)}\int_{B_q(R)}|\nabla \phi|^2\ \d v_g \geq \frac{C_s}{R^2}\left(\frac{1}{V_q(R)}\int_{B_q(R)} \phi^{2 \nu_2} \ \d v_g\right)^{1/\nu_2}
\end{equation}
for all $\phi\in H_c^{1,2}\left(B_q(R)\right)$. 
Given $0 < \theta < 1$, take a constant $C_v$ such that
\begin{equation}\label{p-li-volume-growth}
    \frac{V_q(R)}{V_q(\theta R)} \leq C_v.
\end{equation}
Then, for any $k>0$, there exists a constant $C_{sub}>0$ depending only on $k, \nu_1, p, C_s$, and $C_v$ such that
\begin{equation*}
\|u\|_{\infty, \theta R} \leq C_{sub}\left(\left(A R^2\right)^{p/(p-\nu_1)}+(1-\theta)^{-2}\right)^{\nu_1/k} \frac{\|u\|_{k, R}}{V_q(R)^{1/k}} ,
\end{equation*}
where
\begin{equation*}
    A = 
    \begin{cases}
        \left(\frac{1}{V_q(R)}\int_{B_q(R)}f^p\ \d v_g\right)^{1/p} & \quad \mbox{if } p < \infty,\\
        \|f\|_{\infty,R} & \quad \mbox{if } p = \infty.
    \end{cases}
\end{equation*}
\end{theorem}
The following mean value inequality is also important in proving the elliptic Harnack inequality.
\begin{theorem}(\cite[Lemma 11.2]{p_li})\label{sup-lemma}
 Let $(M,g)$ be an $n$-dimensional complete Riemannian manifold and $u$ be a non-negative function in $H^{1,2}\left(B_q(R)\right)$ satisfying 
$$
\Delta u \leq A u
$$
for some constant $A \geq 0$. Let $\nu_1=\frac{n}{2}$ when $n>2$, and $1<\nu_1<\infty$ be arbitrary when $n=2$. For $\nu_2 > 0$ satisfying $\frac{1}{\nu_1}+\frac{1}{\nu_2}=1$, we assume 
\eqref{sobolev-constant-ineq} and \eqref{p-li-volume-growth} for $\theta = \frac{1}{16}$ with constants $C_s$ and $C_v$, respectively.
In addition, assume that the first nonzero Neumann eigenvalues $\lambda_1\left(\frac{R}{4}\right)$ and $\lambda_1\left(\frac{R}{2}\right)$ of the balls $B_q\left(\frac{R}{4}\right)$ and $B_q\left(\frac{R}{2}\right)$ satisfy the estimate
\begin{equation}\label{p-li-neu-poincare}
    \min \left\{\frac{R^2}{16} \lambda_1\left(\frac{R}{4}\right), \frac{R^2}{4} \lambda_1\left(\frac{R}{2}\right)\right\} \geq C_{np}
\end{equation}
for some constant $C_{np}>0$.
Then, for $k>0$ sufficiently small, there exists a constant $C_{sup}>0$ depending only on the quantities $k, \nu_1, C_v, C_{np}, C_s$, and $\left(A R^2+1\right)$ such that
\begin{equation}\label{sup-mean-ineq}
\frac{\|u\|_{k, R/8}}{V_q(R/8)^{1/k}} \leq C_{sup} \inf _{B_q\left(R/16\right)} u .
\end{equation}
\end{theorem}
From Theorem \ref{sub-lemma} and Theorem \ref{sup-lemma}, we derive the following elliptic Harnack inequality.
\begin{theorem}(\cite[Theorem 11.1]{p_li})\label{elliptic-harnack}
 Let $(M,g)$ be an $n$-dimensional complete Riemannian manifold and $u$ be a non-negative function in $H^{1,2}\left(B_q(R)\right)$ such that
$$
|\Delta u| \leq A u
$$
for some $A\geq 0$.
We assume \eqref{sobolev-constant-ineq}, \eqref{p-li-volume-growth} for $\theta = \frac{1}{16}$ and \eqref{p-li-neu-poincare} with constants $C_s,C_v$ and $C_{np}$, respectively.
Then, there exists a constant $C_{H}>0$ depending on the quantities $\left(A R^2+1\right),\nu_1, n, C_{np}$, $C_s$, and $C_v$ such that
$$
\sup _{B_q\left(R/16\right)} u \leq C_{H} \inf _{B_q\left(R/16\right)} u.
$$
\end{theorem}
We recall the following local Poincar\'{e} type inequality, which we use later in this paper to deduce the modified local Sobolev inequality (Theorem \ref{sobolev-eps-range}) from the local Sobolev inequality.
\begin{theorem}(\cite[Lemma 6.1 in Chapter II]{yau-1}, Local Poincar\'{e} inequality)\label{yau-poincare-thm}
    Let $(M,g)$ be an  $n$-dimesional complete Riemannian manifold. We assume $\operatorname{Ric}_g \geq - K$ with $K\geq 0$. 
    Then, there exist positive constants $C_1$ and $C_2$ depending only on $p \geq 1$ and $n$ such that
$$
\int_{B_q(R)}|\varphi|^p \ \d v_g \leq C_1 R^p \e^{C_2 \sqrt{K} R} \int_{B_q(R)}|\nabla \varphi|^p\  \d v_g
$$
for all $\varphi \in C_0^{\infty}(B_q(R))$, where $C_0^{\infty}(B_q(R))$ denotes compact support elements of $C^{\infty}(B_q(R))$.
\end{theorem}
\subsection{Analysis of harmonic functions}
In \cite{wu}, Wu obtained the following Liouville type theorem, which was proved by using the Li-Yau gradient estimate and his method cannot be directly applied to the range $N\leq 1$. 
\begin{theorem}(\cite[Corollary 3.4]{wu})\label{wu-upper-bound-liouville}
    Let $(M,g,\mu)$ be an $n$-dimensional complete weighted Riemannian manifold, $N\in[n,\infty)$ and $u$ be a smooth positive function satisfying $\Delta_\psi u = 0$. We assume $\Ric_\psi^N\geq 0$. Then, $u$ is necessarily constant.
\end{theorem}
In the rest of this subsection, we list known results concerning Liouville type theorems and a gradient estimate of harmonic functions.
The following $L^p$-Liouville theorem was obtained in \cite{li-schoen}. 
\begin{theorem}(\cite[Theorem 2.5]{li-schoen})\label{lp-liouville}
    Let $(M,g)$ be an $n$-dimensional complete Riemannian manifold. Then, there exists a constant $\delta > 0$ depending on $n$ such that the following property holds:

    If, for some $q\in M$, we have
    \begin{equation*}
        \Ric_g \geq -\delta d(q,x)^{-2}
    \end{equation*}
    when $d(q,x)$ is sufficiently large, 
    then any non-negative $L^p$-function $u$ with $p\in (0,\infty)$ satisfying $\Delta u\geq 0$ is necessarily constant.
\end{theorem}
The weighted case of this theorem is the following theorem. In Section 4, we give two different proofs of this theorem, which are essentially different from that in \cite{jy-1}.
\begin{theorem}\label{wu-lp-liouville}(\cite[Theorem 6.1]{jy-1})
    Let $(M,g,\mu)$ be an $n$-dimesional complete weighted Riemannian manifold. Assume that $|\psi| < A$ for some positive constant $A > 0$. Then, there exists a constant $\delta > 0$ depending on $n$ and $A$ such that the following property holds:

    If, for some $q\in M$, we have
    \begin{equation*}
        \Ric_\psi^\infty \geq -\delta d(q,x)^{-2}
    \end{equation*}
    when $d(q,x)$ is sufficiently large, 
    then any non-negative $L^p (\mu)$-function $u$ with $p\in (0,\infty)$ satisfying $\Delta_\psi u\geq 0$ is identically zero.
\end{theorem}
See Remark \ref{delta-n-a} for the dependence of $\delta$ on $A$.
The following Liouville type theorem does not need any assumptions on the Ricci curvature, which we generalize to the weghted case in Theorem \ref{lp-liouville-1}.
\begin{theorem}(\cite[Theorem 6.3]{yau-1})
    Let $(M,g)$ be an $n$-dimensional complete Riemannian manifold. For $p > 1$, let $u$ be a smooth positive $L^p$-function satisfying $\Delta u \geq 0$. Then $u$ is necessarily constant.
\end{theorem}
We also have the following Liouville theorem, which is obtained in \cite{munteau_wang-1} by combining the mean value inequality and the Bochner formula.
We generalize this theorem in Section 5.
\begin{theorem}(\cite[Theorem 3.2]{munteau_wang-1})\label{wang-sublinear}
    Let $(M,g,\mu)$ be an $n$-dimensional complete weighted Riemannian manifold with $\Ric_{\psi}^{\infty} \geq 0$ and $\psi$ bounded. Then, any $\psi$-harmonic function of sublinear growth is necessarily constant.
\end{theorem}
Here, a $\psi$-harmonic function $u$ is said to be of sublinear growth if 
\begin{equation*}
    \lim_{d(q,x)\rightarrow \infty}\frac{|u(x)|}{d(q,x)} = 0
\end{equation*}
for some $q\in M$.
Combining the elliptic Harnack inequality and the Bochner formula, 
the following gradient estimate for $\psi$-harmonic functions was obtained in \cite{soliton}.
We will present a similar result in Section 6 under lower bounds of $\Ric_\psi^N$ with $\eps$-range and boundedness weight function.
\begin{theorem}(\cite[Theorem 3.1]{soliton})\label{grad-est-thm}
    Let $(M,g,\mu)$ be an $n$-dimensional complete weighted Riemannian manifold with $\Ric_\psi^{\infty} \geq-(n-1)$. Assume that there exists a constant $h > 0$ such that 
   $$
   \sup _{y \in B_x(1)}|\psi(y)-\psi(x)| \leq h 
   $$
   for any $x \in M$.
   Let $u$ be a positive $\psi$-harmonic function.
   Then, there exists a constant $C(n, h)$ depending only on $n$ and $h$ such that we have
   $$
   |\nabla \log u| \leq C(n, h) .
   $$
   \end{theorem}

\section{Functional inequalities}
The goal of this section is to show the mean value inequality under lower bounds of $\Ric_\psi^N$ with $\eps$-range.
\subsection{Sobolev inequality}
In this subsection, we generalize the local Sobolev inequality and the local Poincar\'{e} inequality.
Using these functional inequalities, we obtain the modified local Sobolev inequality (Theorem \ref{sobolev-eps-range}).
First, we show the following Neumann-Poincar\'{e} inequality, which will be used to derive the local Sobolev inequality.
\begin{theorem}\label{modified-neumann-poincare-thm}
        Let $(M,g,\mu)$ be an $n$-dimensional complete weighted Riemannian manifold with $N \in (-\infty,1] \cup [n,+\infty]$,
        $\ez \in \R$ in the $\ez$-range \eqref{epsilin-range} and $b \ge a>0$.
        Assume that
        \begin{equation}
        a \le \e^{-\frac{2(\ez-1)}{n-1}\psi} \le b.
        \end{equation} 
        Then, we have 
        \begin{equation}\label{local-poincare-ineq}
            \int_{B_q(R)}|\varphi - \varphi_{B_q(R)}|^2 \d \mu \leq 2^{n + 3} \left(\frac{2b}{a}\right)^{\frac{1}{c}} \exp\left( \sqrt{\frac{K_\eps(q,2R)}{c}} \frac{2R}{a} \right) R^2 \int_{B_q(2R)}|\nabla \varphi |^2 \ \d \mu 
        \end{equation}
        for any $\varphi\in C_0(B_q(R))$.
\end{theorem}
\underline{\emph{Sketch of proof}}

We apply the argument in \cite[Theorem 7]{yasuaki}.
For any $x,y\in M$, let
\begin{equation*}
   \gamma_{x,y}:[0,d(x,y)] \rightarrow M
\end{equation*}
be a geodesic from $x$ to $y$ parameterized by arclength. Then, we set 
\begin{equation*}
   l_{x,y}(t) = \gamma_{x,y}(td(x,y))
\end{equation*}
for $t\in[0,1]$ and let $J_{x,t}$ be the Jacobian of the map $\Phi_{x,t}: y\mapsto l_{x,y}(t)$. 
By the same argument as in \cite[Theorem 7]{yasuaki}, we obtain the following inequality:
\begin{equation}\label{modified-sobolev-eq-2}
   \int_{B_q(R)}|\varphi - \varphi_{B_q(R)}|^2 \d \mu \leq  2 (2R)^2 F(R)  \int_{B_q(2R)}|\nabla \varphi (z)|^2 \ \d \mu (z),
\end{equation}
where $F(R)$ is the function such that 
\begin{equation}\label{j-lower-bound}
   \forall x,y \in B_q(R),\forall t\in [1/2,1],\quad J_{x,t}(y) \geq \frac{1}{F(R)}.
\end{equation}
Applying the property \eqref{non_increasing_property} with $K = K_\eps(q,2R)$ to the argument in \cite[Theorem 7]{yasuaki}, we get
\begin{equation}\label{modified-sobolev-eq-1}
   \forall x,y \in B_q(R),\forall t\in [1/2,1],\quad J_{x,t}(y) \geq \left(\frac{1}{2}\right)^n\left(\frac{a}{2b}\right)^{1/c}\exp\left(-\sqrt{\frac{K_\eps(q,2R)}{c}}\frac{2R}{a}\right).
\end{equation}
Combining \eqref{modified-sobolev-eq-2} with \eqref{modified-sobolev-eq-1}, we obtain the desired inequality \eqref{local-poincare-ineq}.
\qed

By modifying the argument in \cite[Theorem 8]{yasuaki}, we have the following local Sobolev inequality.
\begin{theorem}(Local Sobolev inequality)\label{sobolev_thm_revisited}
    Under the same assumptions as Theorem \ref{modified-neumann-poincare-thm},
    there exist positive constants $\widetilde{D}_1, \widetilde{E}_1$ depending on $c,a,b, n$ such that 
    \begin{equation}\label{modified-local-sobolev}
    \left(V_q(R)^{-1}\int_{B_q(R)}|\varphi|^{\frac{2\nu}{\nu-2}}\ \d\mu\right)^{\frac{\nu-2}{\nu}}\leq \widetilde{E}_1\exp\left(\widetilde{D}_1\sqrt{K_{\eps}(q,10R)}R\right)R^2 V_q(R)^{-1}\int_{B_q(R)}(|\nabla \varphi|^2 + R^{-2} \varphi^2)\d\mu
    \end{equation}
    for all $B_q(R)\subset M$ and $\varphi\in C_0^{\infty}(B_q(R))$, where we set
    \begin{equation}\label{def-of-nu}
        \nu = \begin{cases}
            3 & \quad \mbox{if } c = 1,\\
            1 + \frac{1}{c} &\quad \mbox{if } c < 1 .
        \end{cases}
    \end{equation}
 \end{theorem}
We remark that we have $c = 1$ only when $n = 2$ and $\eps = 0$, and $c < 1$ otherwise. 
We set $\nu = 3$ for $c = 1$ since we need $\nu > 2$ in the following argument.
We first present two lemmas before proving Theorem \ref{sobolev_thm_revisited}.
\begin{lemma}\label{claim-1}
Under the same assumption as Theorem \ref{modified-neumann-poincare-thm},
we have
    \begin{equation}\label{claim-1-eq}
        \|\varphi_s\|_2 \leq \left(\frac{b}{a}\right)^{3 + \frac{3}{2c}}2^{\frac{1}{2}\left(1 + \frac{1}{c}\right)}\left(\frac{4R}{s}\right)^{\frac{1}{2}\left(1 + \frac{1}{c}\right)}\exp\left(\sqrt{\frac{K_\eps(q,10R)}{c}}\frac{9R}{2a}\right)\frac{1}{\sqrt{V_q(R)}}\|\varphi\|_1
    \end{equation}
    for all $0 < s < R$ and $\varphi\in C_0^{\infty}(B_q(R))$, where we set 
    \begin{equation*}
         \chi_s (x,z) = \frac{1}{V_x(s)}1_{B_x(s)}(z)  \quad \mbox{and} \quad \varphi_s(x) = \int \chi_s(x,z) \varphi(z)\  \d\mu(z).
    \end{equation*}
\end{lemma}
\begin{proof}
We apply the argument in \cite[Lemma 1]{yasuaki}.
According to Theorem \ref{bishop_gromov}, for $x\in M$, we have
\begin{eqnarray}\label{vol-comparison-in-lem-1}
\frac{V_x(r)}{V_x(s)} \leq \frac{b}{a}\frac{\int_0^{r/a}\bs_{-cK}(\tau)^{1/c} \ \d \tau}{\int_0^{s/b}\bs_{-cK}(\tau)^{1/c}\ \d \tau} \leq  \left(\frac{b}{a}\right)^{2 + \frac{1}{c}}\left(\frac{r}{s}\right)^{1 + \frac{1}{c}}\exp\left(\sqrt{\frac{K}{c}}\frac{r}{a}\right)\label{lemma-1-eq-0}
\end{eqnarray}
with $K = K_\eps(x,r)$, where the last inequality follows from the direct calculation. 

We first estimate $\|\varphi_s\|_1$. 
For $x\in \supp\varphi \subset B_q(R)$ and $z\in \supp \chi_s(x, \cdot)$, we have $d(z,x) < s$.
Hence, we find
\begin{eqnarray}
    V_z(s) &\leq & V_x(2s)\nonumber\\
    &\leq & V_x(s)\left(\frac{b}{a}\right)^{2 + \frac{1}{c}} 2^{1 + \frac{1}{c}} \exp\left( \sqrt{\frac{K_\eps(x,2s)}{c}}\frac{2s}{a}\right)\nonumber\\
    &\leq & V_x(s)\left(\frac{b}{a}\right)^{2 + \frac{1}{c}} 2^{1 + \frac{1}{c}} \exp\left( \sqrt{\frac{K_\eps(q,10R)}{c}}\frac{2s}{a}\right)\label{doubling-property},
\end{eqnarray}
where we used $B_x(2s)\subset B_q(3R) \subset B_q(10R)$ in the last inequality. Using \eqref{doubling-property}, we  obtain
\begin{equation*}
    \chi_s (x,z) = \frac{1}{V_x(s)}1_{B_x(s)}(z) \leq \frac{1}{V_z(s)} \left(\frac{b}{a}\right)^{2 + \frac{1}{c}} 2^{1 + \frac{1}{c}} \exp\left( \sqrt{\frac{K_\eps(q,10R)}{c}}\frac{2s}{a}\right) 1_{B_x(s)}(z)
\end{equation*}
for $x\in \supp \varphi$ and $z\in \supp \chi_s(x,\cdot)$.
Therefore, we get
\begin{eqnarray}
    \|\varphi_s\|_1 
    &=& \int \ \d\mu(x) \left| \int \chi_s(x,z)\varphi(z) \ \d \mu(z) \right|\nonumber \\
    &\leq& \int \ \d\mu(x) \int \frac{1}{V_z(s)} \left(\frac{b}{a}\right)^{2 + \frac{1}{c}} 2^{1 + \frac{1}{c}} \exp\left( \sqrt{\frac{K_\eps(q,10R)}{c}}\frac{2s}{a}\right) 1_{B_x(s)}(z) |\varphi(z)|\ \d\mu(z)\nonumber \\
    &=& \left(\frac{b}{a}\right)^{2 + \frac{1}{c}} 2^{1 + \frac{1}{c}} \exp\left( \sqrt{\frac{K_\eps(q,10R)}{c}}\frac{2s}{a}\right) \|\varphi\|_1\label{lemma-1-eq-1},
\end{eqnarray}
where we used $\int 1_{B_x(s)}(z)\ \d\mu(x) = V_z(s)$ in the last equality.

We next estimate $\|\varphi_s\|_{\infty}$. For $x\in \supp\varphi_s$, we have $B_q(R)\cap B_x(s)\neq \emptyset$. From the volume growth estimate \eqref{lemma-1-eq-0} and $B_x(3R)\subset B_q(10R)$, we obtain
\begin{eqnarray}
    \frac{V_x(2R + s)}{V_x(s)} 
    &\leq & \left(\frac{b}{a}\right)^{2 + \frac{1}{c}}\left( \frac{2R + s}{s} \right)^{1 + \frac{1}{c}}\exp\left(\sqrt{\frac{K_\eps(q,10R)}{c}}\frac{2R + s}{a} \right)\label{modified-sobolev-eq-3}.
\end{eqnarray} 
Since $B_x(4R) \subset B_q(10R)$, we also have 
\begin{eqnarray}
    \frac{V_x(4R)}{V_x(2R + s)} 
    &\leq& \left(\frac{b}{a}\right)^{2 + \frac{1}{c}}\left(\frac{4R}{2R + s}\right)^{1 + \frac{1}{c}}\exp\left(\sqrt{\frac{K_\eps(q,10R)}{c}}\frac{4R}{a}\right)\label{modified-sobolev-eq-4}.
\end{eqnarray}
Combining \eqref{modified-sobolev-eq-3} with \eqref{modified-sobolev-eq-4}, we have 
\begin{eqnarray*}
    \frac{1}{V_x(s)}  &\leq & \left(\frac{b}{a}\right)^{2\left(2 + \frac{1}{c}\right)}\left(\frac{4R}{s}\right)^{1 + \frac{1}{c}}\exp\left(\sqrt{\frac{K_\eps(q,10R)}{c}} \frac{6R + s}{a} \right)\frac{1}{V_x(4R)}\\
    &\leq & \left(\frac{b}{a}\right)^{2\left(2 + \frac{1}{c}\right)}\left(\frac{4R}{s}\right)^{1 + \frac{1}{c}}\exp\left(\sqrt{\frac{K_\eps(q,10R)}{c}} \frac{6R + s}{a} \right) \frac{1}{V_q(R)},
\end{eqnarray*}
where we used $B_q(R) \subset B_x(4R)$ in the latter inequality. Hence, we get
\begin{eqnarray}
    \|\varphi_s\|_\infty 
    &=& \left\| \int \frac{1}{V_x(s)} 1_{B_x(s)} \varphi(z)\ \d\mu(z) \right\|_{\infty}\nonumber\\
    &\leq& \left(\frac{b}{a}\right)^{2\left(2 + \frac{1}{c}\right)}\left(\frac{4R}{s}\right)^{1 + \frac{1}{c}}\exp\left( \sqrt{\frac{K_\eps(q,10R)}{c}}\frac{6R + s}{a} \right)\frac{1}{V_q(R)}\left\| \int 1_{B_x(s)} \varphi(z)\ \d \mu(z) \right\|_{\infty}\nonumber\\
    &\leq& \left(\frac{b}{a}\right)^{2\left(2 + \frac{1}{c}\right)}\left(\frac{4R}{s}\right)^{1 + \frac{1}{c}}\exp\left( \sqrt{\frac{K_\eps(q,10R)}{c}}\frac{6R + s}{a} \right)\frac{\|\varphi\|_1}{V_q(R)}\label{lemma-1-eq-2}.
\end{eqnarray}
From \eqref{lemma-1-eq-1} and \eqref{lemma-1-eq-2}, we have 
\begin{eqnarray*}
    \|\varphi_s\|_2 &\leq &\sqrt{\|\varphi_s\|_{\infty}}\sqrt{\|\varphi_s\|_1}\\
    &\leq& \left(\frac{b}{a}\right)^{3 + \frac{3}{2c}}2^{\frac{1}{2}\left(1 + \frac{1}{c}\right)}\left(\frac{4R}{s}\right)^{\frac{1}{2}\left(1 + \frac{1}{c}\right)}\exp\left(\sqrt{\frac{K_\eps(q,10R)}{c}}\frac{6R+3s}{2a}\right)\frac{1}{\sqrt{V_q(R)}}\|\varphi\|_1.
\end{eqnarray*}
Therefore, we obtain the desired inequality.
\end{proof}
For a ball $B = B_q(R)$ and $k > 0$, $kB$ denotes the ball $B_q(kR)$ in the following arguments.
\begin{lemma}\label{claim-2}
    Under the same assumption as Theorem \ref{modified-neumann-poincare-thm}, we have
    \begin{equation}\label{claim-2-desired-eq}
        \|\varphi - \varphi_s\|_2 \leq 2^{8 + \frac{5}{c} + \frac{n}{2}}\left(\frac{b}{a}\right)^{3 + \frac{2}{c}}\exp\left(\sqrt{\frac{K_\eps(q,10R)}{c}}\frac{11R}{a}\right) s \|\nabla \varphi \|_2
    \end{equation}
    for all $0 < s < R$ and $\varphi \in C_0^{\infty}(B_q(R))$.
\end{lemma}
\begin{proof}
We fix $0 < s < R$. 
By setting the order of sets by inclusion and using Zorn's lemma, we see that there exists a set $X\subset B_q(2R)$ satisfying the following (a) and (b):
\begin{itemize}
    \item[(a)] For distinct $x_i,x_j\in X$, we have $B_{x_i}(s/2)\cap B_{x_j}(s/2) = \emptyset$.
    \item[(b)] For any $x'\in B_q(2R)\setminus X$, there exists $x\in X$ such that $B_x(s/2)\cap B_{x'}(s/2)\neq \emptyset$.
\end{itemize}
We note that (b) indicates the maximality of $X$ as a set satisfying (a).
Labelling the elements in $X$ as $X = \{x_j\in X\ |\ j \in J_q\}$, we denote $B_j = B_{x_j}(s/2)$ for $j\in J_q$.
We see that $\{2B_j\ |\ j\in J_q\}$ is an open covering of $B_q(2R)$. Indeed, for $y\in X\subset B_q(2R)$, we have $y \in \bigcup_{j\in J_q} 2B_j$ trivially. For $y \in B_q(2R)\setminus X$, 
there exists $x_i\in X$ such that $B_y(s/2)\cap B_{x_i}(s/2)\neq \emptyset $. Hence, for $w\in B_y(s/2)\cap B_{x_i}(s/2)$, we have $d(x_i,y)\leq d(x_i,w) + d(w,y) < s$ and hence, $y\in B_{x_i}(s)$.

We conduct the argument in \cite[Lemma 2]{yasuaki} on $B_q(2R)$. 
For $z\in B_q(2R)$, let $J_q(z) = \{i \in J_q\  |\  z \in 8B_i\}$ and $N(z) = \# J_q(z)$.
We first estimate $N(z)$ from above.
Let $B_z$ be a ball in $\{ B_j\ |\ j\in J_q \}$ such that $z\in 2B_z$. For $i\in J_q(z)$, we have $B_z \subset 16 B_i$. Hence, from the volume growth estimate \eqref{vol-comparison-in-lem-1} and $16 B_i \subset B_{x_i}(8R)$, we have 
\begin{eqnarray}
    \mu(B_z) \leq \mu(16B_i) \leq   \left(\frac{b}{a}\right)^{2 + \frac{1}{c}} 16^{1 + \frac{1}{c}}\exp\left( \sqrt{\frac{K_\eps(x_i,8R)}{c}} \frac{8R}{a} \right) \mu(B_i)\label{lemma-2-eq-0}.
\end{eqnarray}
Since $d(q,x_i)  \leq 2R$, we have $B_{x_i}(8R) \subset B_q(10R)$. Hence, we obtain
\begin{equation*}
    \mu(B_z) \leq C_1 \mu(B_i),
\end{equation*}
where 
\begin{equation}\label{lemma-2-eq-1}
    C_1 =  \left(\frac{b}{a}\right)^{2 + \frac{1}{c}} 16^{1 + \frac{1}{c}}\exp\left( \sqrt{\frac{K_\eps(q,10R)}{c}} \frac{8R}{a} \right).
\end{equation}
Therefore, we have
\begin{equation}\label{lemma-2-eq-1-5}
\sum_{i\in J_q(z)}\mu(B_i) \geq N(z)\frac{\mu(B_z)}{C_1}.
\end{equation}
On the other hand, for $i\in J_q(z)$, we have $B_i \subset 16 B_z$. 
Let $y\in M$ be the center of $B_z$. 
Since $d(q,y) \leq 2R$, we have $B_y(8R)\subset B_q(10R)$. Hence, we find
\begin{eqnarray}
    \sum_{i\in J_q(z)} \mu(B_i) &\leq& \mu(16 B_z) \leq  \left(\frac{b}{a}\right)^{2 + \frac{1}{c}} 16^{1 + \frac{1}{c} }\exp\left( \sqrt{\frac{K_\eps(y, 8R)}{c}} \frac{8R}{a} \right)  \mu(B_z) \leq C_1 \mu(B_z)\label{lemma-2-eq-2}.
\end{eqnarray}
Combining this with \eqref{lemma-2-eq-1-5}, we get
\begin{equation*}
    N(z)\frac{\mu(B_z)}{C_1} \leq C_1\mu(B_z).
\end{equation*}
Therefore, setting $N_q = C_1^2$, we have $N(z) \leq N_q$.
Since
\begin{equation}\label{modified-sobolev-eq-5}
    \|\varphi - \varphi_s\|_2^2\leq \sum_{i\in J_q}\left(2\int_{2B_i}|\varphi(x) - \varphi_{4B_i}|^2 + |\varphi_{4B_i} - \varphi_s(x)|^2\ \d\mu(x)\right),
\end{equation}
we estimate the right-hand side of \eqref{modified-sobolev-eq-5} to obtain the desired inequality \eqref{claim-2-desired-eq}.
Since $d(q,x_i)  \leq 2R$ for $i\in J_q$, we have $B_{x_i}(4s) \subset B_q(10R)$.
Combining this with the Neumann-Poincar\'{e} inequality \eqref{local-poincare-ineq}, we have 
\begin{eqnarray}\label{modified-sobolev-eq-6}
    \int_{4B_i} |\varphi - \varphi_{4B_i} |^2 \ \d \mu &\leq & 2^{n+3}\left(\frac{2b}{a}\right)^{\frac{1}{c}}\exp\left( \sqrt{\frac{K_\eps(x_i, 4s)}{c}}\frac{4s}{a} \right) (2s)^2 \int_{8B_i}|\nabla \varphi|^2 \ \d\mu\nonumber\\
     &\leq& C_2 s^2 \int_{8B_i}|\nabla \varphi|^2 \ \d\mu,
\end{eqnarray}
where 
\begin{equation*}
    C_2 = 2^{n+5}\left(\frac{2b}{a}\right)^{\frac{1}{c}}\exp\left( \sqrt{\frac{K_\eps(q, 10R)}{c}}\frac{4R}{a} \right) .
\end{equation*}
From the volume growth estimate \eqref{vol-comparison-in-lem-1} and $B_{x_i}(2R)\subset B_q(10R)$, for any $x\in 2B_i = B_{x_i}(s)$, we have
\begin{eqnarray*}
    V_{x_i}(s) &\leq& V_x(2s)\\
    &\leq& V_x(s)\left(\frac{b}{a}\right)^{2 + \frac{1}{c}}2^{1 + \frac{1}{c}} \exp\left( \sqrt{\frac{K_\eps(x_i,2R)}{c}} \frac{2s}{a} \right)\\
    &\leq& V_x(s)\left(\frac{b}{a}\right)^{2 + \frac{1}{c}}2^{1 + \frac{1}{c}} \exp\left( \sqrt{\frac{K_\eps(q,10R)}{c}} \frac{2s}{a} \right).
\end{eqnarray*}
Hence, we obtain
\begin{eqnarray}
    \int_{2B_i}\left|\varphi_{4B_i} - \varphi_s(x)\right|^2 \ \d\mu(x) 
    &=& \int_{2B_i}\left|\int_{B_x(s)}\frac{1}{V_x(s)}(\varphi_{4B_i} - \varphi(z))\ \d\mu(z)\right|^2\d\mu(x)\nonumber\\
    &\leq & \int_{2B_i}\ \d \mu(x)\int_{B_x(s)} \frac{1}{V_x(s)} \left|\varphi_{4B_i} - \varphi(z) \right|^2\ \d\mu(z) \nonumber\\
    &\leq&\frac{1}{V_{x_i}(s)}\left(\frac{b}{a}\right)^{2 + \frac{1}{c}}2^{\frac{1}{c}+ 1}\exp\left(\sqrt{\frac{K_\eps(q,10R)}{c}}\frac{2s}{a}\right) \int_{2B_i}\ \d\mu(x)\int_{4B_i}|\varphi_{4B_i} - \varphi(z)|^2\ \d\mu(z)\nonumber\\
    &\leq & \left(\frac{b}{a}\right)^{2 + \frac{1}{c}}2^{\frac{1}{c}+ 1}\exp\left(\sqrt{\frac{K_\eps(q,10R)}{c}}\frac{2R}{a}\right)C_2s^2\int_{8B_i}|\nabla \varphi|^2\ \d\mu,\label{modified-sobolev-eq-7}
\end{eqnarray}
where we used $B_{x}(s)\subset 4B_i$ for $x\in 2B_i$ in the second inequality and \eqref{modified-sobolev-eq-6} in the last inequality.
Combining \eqref{modified-sobolev-eq-5}, \eqref{modified-sobolev-eq-6} and \eqref{modified-sobolev-eq-7}, we get
\begin{equation*}
    \|\varphi - \varphi_s\|^2_2\leq C_{3}s^2\sum_{i\in J_q}\int_{8B_i}|\nabla \varphi|^2 \ \d\mu \leq C_{3}N_qs^2\|\nabla \varphi\|^2_2,
\end{equation*}
where
\begin{equation*}
    C_3 = 4\left(\frac{b}{a}\right)^{2 + \frac{1}{c}}2^{\frac{1}{c}+ 1}\exp\left(\sqrt{\frac{K_\eps(q,10R)}{c}}\frac{2R}{a}\right)C_2 \geq 2\left\{\left(\frac{b}{a}\right)^{2 + \frac{1}{c}}2^{\frac{1}{c}+ 1}\exp\left(\sqrt{\frac{K_\eps(q,10R)}{c}}\frac{2R}{a}\right)C_2 + C_2\right\}.
\end{equation*}
Hence, we obtain
\begin{equation*}
    \|\varphi - \varphi_s\|_2 \leq \sqrt{N_q C_3} s \|\nabla \varphi\|_2.
\end{equation*}
Since
\begin{equation*}
    \sqrt{N_q C_3} = 2^{8 + \frac{5}{c} + \frac{n}{2}}\left(\frac{b}{a}\right)^{3 + \frac{2}{c}}\exp\left(\sqrt{\frac{K_\eps(q,10R)}{c}}\frac{11R}{a}\right),
\end{equation*}
we have the desired inequality \eqref{claim-2-desired-eq}.
\end{proof}

\underline{\textit{Sketch of proof of Theorem \ref{sobolev_thm_revisited}}}

We apply the argument in \cite[Theorem 8]{yasuaki}.
For $0 < s \leq R$ and $\varphi\in C_0^{\infty}(B_q(R))$, we have
\begin{equation*}
\|\varphi\|_2\leq \|\varphi - \varphi_s\|_2 + \|\varphi_s\|_2.
\end{equation*}
From Lemma \ref{claim-1} and Lemma \ref{claim-2} and recalling $\nu \geq 1 + \frac{1}{c}$, we deduce
\begin{equation*}
\|\varphi\|_2\leq C_5 s\|\nabla \varphi\|_2 + C_4V_q(R)^{-\frac{1}{2}}\left(\frac{R}{s}\right)^{\frac{\nu}{2}}\|\varphi\|_1,
\end{equation*}
where we set
\begin{equation*}
    C_4 =  \left(\frac{b}{a}\right)^{3 + \frac{3}{2c}}8^{\frac{1}{2}\left(1 + \frac{1}{c}\right)}\exp\left(\sqrt{\frac{K_\eps(q,10R)}{c}}\frac{9R}{2a}\right)
\end{equation*}
and 
\begin{equation*}
    C_5 = 2^{8 + \frac{5}{c} + \frac{n}{2}}\left(\frac{b}{a}\right)^{3 + \frac{2}{c}}\exp\left(\sqrt{\frac{K_\eps(q,10R)}{c}}\frac{11R}{a}\right).
\end{equation*}
Hence, we obtain
\begin{equation}\label{robinson}
\|\varphi\|_2\leq 4C_5s\left(\|\nabla \varphi\|_2 + \frac{1}{R}\|\varphi\|_2\right) + C_4V_q(R)^{-\frac{1}{2}}\left(\frac{R}{s}\right)^{\frac{\nu}{2}}\|\varphi\|_1.
\end{equation}
Optimizing the right hand side of \eqref{robinson} over $s > 0$ as in \cite[Theorem 8]{yasuaki}, we have 
\begin{equation}\label{lemma-2-eq-3}
    \|\varphi\|_2 \leq \left\{4C_5C_6^{\frac{2}{2 + \nu}} + C_4C_6^{\frac{-\nu}{2 + \nu}}\right\}\left(\|\nabla \varphi\|_2 + \frac{1}{R}\|\varphi\|_2\right)^{\frac{\nu}{2 + \nu}}V_q(R)^{-\frac{1}{\nu + 2}}R^{\frac{\nu}{\nu + 2}}\|\varphi\|_1^{\frac{2}{2 + \nu}},
\end{equation}
where $C_6 = \frac{\nu}{2}\frac{C_4}{4C_5}$.
Since
\begin{equation*}
C_5C_6^{\frac{2}{2 + \nu}} = E_0\exp\left(\sqrt{\frac{K_\eps(q,10R)}{c}}\left(\frac{11R}{a} - \frac{13R}{2a}\frac{2}{2 + \nu}\right)\right)
\end{equation*}
and
\begin{equation*}
C_4C_6^{\frac{-\nu}{2 + \nu}} = E_1\exp\left(\sqrt{\frac{K_\eps(q,10R)}{c}}\left(\frac{9R}{2a} - \frac{13R}{2a}\frac{-\nu}{2 + \nu}\right)\right),
\end{equation*}
where $E_0, E_1$ are constants depending on $c,a,b,n$,
 we see that there exist constants $D_2,E_2$ depending on $c,a,b,n$ such that 
\begin{equation}\label{lemma-2-eq-4}
    4C_5C_6^{\frac{2}{2 + \nu}} + C_4C_6^{\frac{-\nu}{2 + \nu}} < E_2\exp\left( D_2\sqrt{K_\eps(q,10R)}R\right).
\end{equation}
Since $\nu > 2$, by applying the same argument as in \cite[Theorem 8]{yasuaki} using \eqref{lemma-2-eq-3} and \eqref{lemma-2-eq-4}, we obtain the desired inequality (we also refer to \cite[Remark 1]{yasuaki}).
\qed

We next show the following local Poincar\'{e} type inequality using the Laplacian comparison theorem. This is a generalization of Theorem \ref{yau-poincare-thm}.
\begin{lemma}(Local Poincar\'{e} inequality)\label{yau-poincare-lem}
    Under the same assumption as Theorem \ref{modified-neumann-poincare-thm}, 
     for $p\geq 1$, there exist positive constants $C_7$ and $C_8$ depending only on $c,a,b,p$ such that 
     \begin{equation}\label{poincare}
        \int_{B_q(R)}|\varphi|^p \d \mu \leq C_7 R^p \exp\left({C_8 \sqrt{K_{\eps}(q,5R)} R}\right) \int_{B_q(R)}|\nabla \varphi|^p \ \d \mu
     \end{equation}
     for all $\varphi\in C_0^{\infty}(B_q(R))$.
\end{lemma}
\begin{proof}
    We denote $K = K_{\eps}(q,5R)$.
    We fix a point $x_1$ on $\partial B_q(3 R)$ and set $\rho_1(x)=d(x, x_1)$. By Theorem \ref{laplacian_comparison_theorem}, we have
    \begin{eqnarray*}
        \Delta_\psi\rho_1 &\leq& \frac{\sqrt{K}}{a\sqrt{c}}\coth\left(\frac{\sqrt{cK}}{b}\rho_1\right)\leq\frac{\sqrt{K}}{a\sqrt{c}} + \frac{b}{ac\rho_1}
    \end{eqnarray*}
    on $B_q(5R)$.
    For $x \in B_q(R)$, since $2 R \leq \rho_1(x) \leq 4 R$, we get
    $$
    \Delta_\psi \rho_1 \leq \frac{\sqrt{K}}{a\sqrt{c}} + \frac{b}{2caR} .
    $$
    Setting $\alpha = 2\left(\frac{\sqrt{K}}{a\sqrt{c}} + \frac{b}{2caR}\right)$, we have
    \begin{eqnarray}
    \Delta_\psi \e^{-\alpha \rho_1} =\e^{-\alpha \rho_1}\left(-\alpha \Delta_\psi \rho_1+\alpha^2\right) & \geq&  \frac{\alpha^2}{2} \e^{-\alpha \rho_1}\label{yau-eq-1}.
    \end{eqnarray}
    Hence, for any non-negative $\varphi \in C_0^{\infty}(B_q(R))$, we get
    \begin{equation*}
    \frac{\alpha^2}{2} \int_{B_q(R)} \varphi \e^{-\alpha \rho_1} \ \d \mu  \leq \alpha \int_{B_q(R)} \e^{-\alpha \rho_1} \nabla \varphi \cdot \nabla \rho_1 \ \d \mu  \leq \alpha \int_{B_q(R)} \e^{-\alpha \rho_1}|\nabla \varphi| \ \d \mu,
    \end{equation*}
    where we used \eqref{yau-eq-1} and the integration by parts in the first inequality.
    From $2 R \leq \rho_1(x) \leq 4 R$ for $x\in B_q(R)$, we deduce
    \begin{equation*}
        \frac{\alpha^2}{2}\int_{B_q(R)}\varphi\e^{- 4\alpha R}\ \d \mu \leq \alpha \int_{B_q(R)}\e^{-2\alpha R}|\nabla \varphi|\ \d\mu.
    \end{equation*}
    Since
    \begin{eqnarray*}
        \frac{2}{\alpha}\e^{2\alpha R} &=& \left(\frac{\sqrt{K}}{a\sqrt{c}} + \frac{b}{2caR}\right)^{-1}\exp\left(\frac{4\sqrt{K}}{a\sqrt{c}}R + \frac{2b}{ca}\right)\\
        &\leq&  \frac{2ca}{b}\exp\left( \frac{2b}{ca} \right)R \exp\left( \frac{4\sqrt{K}}{a\sqrt{c}}R \right),
    \end{eqnarray*}
    we obtain
    \begin{equation}\label{eq-1}
        \int_{B_q(R)}\varphi \ \d\mu \leq C_9 R\e^{C_{10}\sqrt{K}R} \int_{B_q(R)}|\nabla\varphi| \ \d\mu,
    \end{equation}
    where we set $C_9 = \frac{2ca}{b}\exp\left( \frac{2b}{ca} \right) $ and $C_{10} =\frac{4}{a\sqrt{c}}.$ 
    We remark that \eqref{eq-1} also holds for any $\varphi\in C_0^{\infty}(B_q(R))$ by replacing $\varphi$ with $|\varphi|$. Since $|\nabla| \varphi ||=| \nabla \varphi |$ almost everywhere, we have 
    \begin{equation}
        \int_{B_q(R)}|\varphi| \ \d\mu \leq C_9 R\e^{C_{10}\sqrt{K}R} \int_{B_q(R)}|\nabla\varphi|\  \d\mu,
    \end{equation}
    which proves the desired inequality for $p= 1$. For $p >  1$, applying \eqref{eq-1} to the function $|\varphi|^p$, we have
    \begin{eqnarray}
        \int_{B_q(R)}|\varphi|^p \ \d \mu &\leq& C_9 R \e^{C_{10} \sqrt{K} R} \int_{B_q(R)} p|\varphi|^{p-1}|\nabla \varphi|\  \d \mu\\
        &\leq& C_9 R \e^{C_{10} \sqrt{K} R} p \left( \int_{B_q(R)} |\varphi|^p \ \d \mu \right)^{\frac{p-1}{p}}\left(\int_{B_q(R)} |\nabla \varphi|^p\ \d \mu\right)^{\frac{1}{p}},
    \end{eqnarray}
    where we used the Hölder's inequality in the second inequality. Therefore, we obtain
    \begin{eqnarray*}
        \int_{B_q(R)} |\varphi|^p \ \d \mu  \leq p^{p} C_9^p R^p \e^{C_{10} p\sqrt{K}R}\left( \int_{B_q(R)} |\nabla \varphi|^p \ \d\mu \right),
    \end{eqnarray*}
    which yields the desired inequality for $p>1$.
\end{proof}

We finally present the following modified local Sobolev inequality, which is a key ingredient in proving the mean value inequality by the Moser's iteration argument.
\begin{theorem}(Modified local Sobolev inequality)\label{sobolev-eps-range}
    Under the same assumptions as Theorem \ref{modified-neumann-poincare-thm},
    there exist positive constants $\widetilde{D}_2, \widetilde{E}_2$ depending on $c,a,b,n$ such that 
    \begin{equation*}
        \left(\frac{1}{V_q(R)}\int_{B_q(R)}|\varphi|^{\frac{2\nu}{\nu-2}}\d\mu\right)^{\frac{\nu-2}{\nu}}\leq \widetilde{E}_2 \exp\left(\widetilde{D}_2\sqrt{K_{\eps}(q,10R)} R\right)\frac{R^2}{V_q(R)}\int_{B_q(R)} |\nabla \varphi|^2\ \d\mu
    \end{equation*}
    for any $\varphi\in C_0^{\infty}(B_q(R))$ and $\nu > 2$ given in \eqref{def-of-nu}.
\end{theorem}
\begin{proof}
    This is proved by combining the local Sobolev inequality and the local Poincar\'{e} inequality.
    Indeed, there exist constants $\widetilde{D}_1$ and $\widetilde{E}_1$ satisfying \eqref{modified-local-sobolev} and constants $C_7$ and $C_8$ satisfying \eqref{poincare} for $p = 2$. Therefore, for arbitrary $\varphi\in C_0^{\infty}(B_q(R))$, we obtain
    \begin{eqnarray*}
        &&\left(\frac{1}{V_q(R)}\int_{B_q(R)}|\varphi|^{\frac{2\nu}{\nu-2}}\d\mu\right)^{\frac{\nu-2}{\nu}}\\
        &&\quad \leq \widetilde{E}_1\exp \left(\widetilde{D}_1\sqrt{K_{\eps}(q,10R)}R\right)\frac{R^2}{V_q(R)}\int_{B_q(R)}(|\nabla \varphi|^2 + R^{-2}\varphi^2)\ \d\mu\\
        &&\quad\leq \widetilde{E}_1\exp\left(\widetilde{D}_1\sqrt{K_{\eps}(q,10R)}R\right)\frac{R^2}{V_q(R)}\left\{\int_{B_q(R)} |\nabla \varphi|^2 \ \d \mu + \frac{1}{R^2}\left(C_7 R^2 \exp\left(C_8 \sqrt{K_{\eps}(q,5R)} R\right)\int_{B_q(R)} |\nabla \varphi|^2\  \d \mu\right)\right\}\\
        &&\quad= \widetilde{E}_1\exp\left(\widetilde{D}_1\sqrt{K_{\eps}(q,10R)}R\right)\left\{1 + C_7 \exp\left(C_8 \sqrt{K_{\eps}(q,5R)} R\right)\right\}\frac{R^2}{V_q(R)}\int_{B_q(R)} |\nabla \varphi|^2\  \d \mu\\
        &&\quad= \widetilde{E}_2 \exp\left(\widetilde{D}_2\sqrt{K_{\eps}(q,10R)}R\right)\frac{R^2}{V_q(R)}\int_{B_q(R)} |\nabla \varphi|^2\ \d\mu,
    \end{eqnarray*}
    where $\widetilde{D}_2, \widetilde{E}_2$ are constants depending on $c,a,b,n$.
\end{proof}

\subsection{Mean value inequality}
Next, we establish a mean value inequality. Using this mean value inequality, we prove $L^p$-Liouville type theorems in Section 4 and a Liouville theorem for harmonic functions of sublinear growth in Section 5.
\begin{theorem} (Mean  value inequality)\label{mean-eq-thm}
    Under the same assumption as Theorem \ref{modified-neumann-poincare-thm},
    let $u$ be a non-negative smooth function satisfying
    \begin{equation*}
        \Delta_\psi u \geq -Au,
    \end{equation*}
    where $A$ is a positive constant.
    Then, for $p > 0$, there exist positive constants $\widetilde{D}_3, \widetilde{E}_3$ depending on $p,c,a,b,n$ such that 
    \begin{equation}\label{mean-value-inequality-k}
        \|u\|_{{\infty},\theta R} \leq \widetilde{E}_3\left\{\exp\left(\widetilde{D}_3 \sqrt{K_{\eps}(q,10R)} R\right) \left(AR^2 + \frac{16}{(1-\theta)^2}\right)\right\}^{\frac{\nu}{2p}} \left(\frac{1}{V_q(\theta R)} \int_{B_q(R)}u^p \d \mu \right)^{\frac{1}{p}}
    \end{equation}
    for all $\theta\in (0,1)$ and $\nu > 2$ in \eqref{def-of-nu}.
\end{theorem}
\begin{proof}
    By Theorem \ref{sobolev-eps-range}, we have the following modified local Sobolev type inequality:
    \begin{equation}\label{spbolev-ingredient}
        \frac{C_s(R)}{R^2}\left( \frac{1}{V_q(R)}\int_{B_q(R)} |\varphi|^{2\beta}\d\mu \right)^{\frac{1}{\beta}} \leq \frac{1}{V_q(R)}\int_{B_q(R)}|\nabla\varphi|^2\d\mu
    \end{equation}
    for any $\varphi\in C_0^{\infty}(B_q(R))$
    where we set $\beta = \frac{\nu}{\nu-2}$ and $C_s(R) = \left\{\widetilde{E}_2\exp\left(\widetilde{D}_2\sqrt{K_{\eps}(q,10R)}R\right)\right\}^{-1}$.
    We fix an arbitrary constant $\gamma \geq 1$.
    For arbitrary $\phi\in C_0^{\infty}(B_q(R))$, we obtain
    \begin{eqnarray}
        &&\int_{B_q(R)}\left|\nabla\left(\phi u^\gamma\right)\right|^2\  \d\mu\\
        &=&\int_{B_q(R)}|\nabla \phi|^2 u^{2 \gamma} \ \d\mu+2 \gamma \int_{B_q(R)} \phi u^{2 \gamma-1}\langle\nabla \phi, \nabla u\rangle \ \d\mu+\gamma^2 \int_{B_q(R)} \phi^2 u^{2 \gamma-2}|\nabla u|^2\ \d\mu\nonumber\\
        &\leq& \int_{B_q(R)}|\nabla \phi|^2 u^{2 \gamma} \ \d\mu+ \gamma \left(2\int_{B_q(R)} \phi u^{2 \gamma-1}\langle\nabla \phi, \nabla u\rangle\  \d\mu+ (2\gamma-1) \int_{B_q(R)} \phi^2 u^{2 \gamma-2}|\nabla u|^2\ \d\mu\right)\nonumber\\
        &= & \int_{B_q(R)} |\nabla \phi|^2 u^{2\gamma} \ \d \mu - \gamma \int_{B_q(R)} \phi^2 u^{2\gamma -1}\Delta_\psi u\  \d \mu\nonumber\\
        &\leq & \int_{B_q(R)} |\nabla \phi|^2 u^{2\gamma} \ \d\mu +\gamma \int_{B_q(R)} \phi^2 A u^{2\gamma} \ \d\mu\label{the-identity},
    \end{eqnarray}
    where we used the integration by parts in the third line and $\Delta_\psi u \geq -A u$ in the last line.
    Combining \eqref{the-identity} with the modified local Sobolev inequality \eqref{spbolev-ingredient}, we get
    \begin{eqnarray}
        \frac{1}{V_q(R)}\left(\gamma\int_{B_q(R)} \phi^2 A u^{2\gamma} \ \d \mu +   \int_{B_q(R)}|\nabla \phi|^2 u^{2 \gamma} \ \d\mu\right) 
        & \geq &\frac{C_s(R)}{R^2}\left(\frac{1}{V_q(R)}\int_{B_q(R)}\left(\phi^2 u^{2 \gamma}\right)^\beta \d\mu\right)^{\frac{1}{\beta}}\label{mean-eq-1}.
    \end{eqnarray}
    Let us choose $\phi\in C_0^{\infty}(B_q(R))$ so that $0\leq \phi \leq 1$, $\phi = 1$ on $ B_q(\rho)$, $\phi = 0$ on $B_q(R) \setminus B_q(\rho+\sigma)$ and $|\nabla\phi|\leq \frac{2}{\sigma}$ on $B_q(\rho+\sigma) \setminus B_q(\rho) $.
    From \eqref{mean-eq-1}, we deduce
    $$
    \begin{aligned}
    \left(\frac{1}{V_q(R)}\int_{B_q(\rho)} u^{2 \gamma \beta} \ \d\mu\right)^{\frac{1}{\beta}} & \leq\left(\frac{1}{V_q(R)}\int_{B_q(R)}\left(\phi^2 u^{2 \gamma}\right)^\beta\  \d\mu\right)^{\frac{1}{\beta}} \\
    & \leq \frac{R^2}{C_s(R)V_q(R)}\left\{ \gamma A\int_{B_q(\rho +\sigma)} u^{2\gamma} \ \d \mu +  \frac{4}{\sigma^2} \int_{B_q(\rho+\sigma)} u^{2 \gamma}\  \d\mu\right\}.
    \end{aligned}
    $$
    Hence, we have
    \begin{eqnarray}\label{iteration-base-1}
        \left(\frac{1}{V_q(R)}\int_{B_q(\rho)}u^{2\gamma \beta}\ \d\mu \right)^{\frac{1}{2\gamma\beta}}\leq  \left\{ \frac{R^2}{C_s(R)}\left(\gamma A + \frac{4}{\sigma^2} \right) \right\}^{\frac{1}{2\gamma}}  \left(\frac{1}{V_q(R)}\int_{B_q(\rho + \sigma)}u^{2\gamma} \ \d\mu \right)^{\frac{1}{2\gamma}}.
    \end{eqnarray}

    We first consider the case $p\geq 2$. We choose the following sequences of $a_i, \rho_i$, and $\sigma_i$ such that
    \begin{equation*}
        \gamma_i=\frac{p \beta^i}{2}, \quad \sigma_i=2^{-(1+i)}(1-\theta) R, \quad \rho_i=R-\sum_{j=0}^i \sigma_j > \theta R 
    \end{equation*}
    for $i \geq 0$.
    We note that, since $\beta > 1$ and $p\geq 2$, we have $\gamma_i\geq 1$. 
    Therefore, we can substitute $\gamma=\gamma_i$, $\rho=\rho_i$, and $\sigma=\sigma_i$ to \eqref{iteration-base-1}. Iterating this procedure, we obtain 
    \begin{equation}\label{will-be-estimated}
        \left(\frac{1}{V_q(R)}\int_{B_q(\rho_i)}u^{2\gamma_i\beta}\ \d \mu\right)^{\frac{1}{2\gamma_i\beta}}\leq \left[\prod_{j = 0}^i\left\{\frac{R^2}{C_s(R)}\left(\gamma_j A + \frac{4}{\sigma_j^2}\right)\right\}^{\frac{1}{2\gamma_j}}\right] \left(\frac{1}{V_q(R)}\int_{B_q(R)}u^{p} \ \d \mu\right)^{\frac{1}{p}}.
    \end{equation}
    We first estimate the left-hand side of \eqref{will-be-estimated}. We have the inequality
    \begin{equation}\label{sub-mean-eq-1}
        \lim_{i \rightarrow \infty} \left(\frac{1}{V_q(R)}\int_{B_q(\rho_i)}u^{2\gamma_i\beta}\ \d\mu\right)^{\frac{1}{2\gamma_i\beta}} \geq \lim_{i\rightarrow \infty}\left(\frac{1}{V_q(R)}\int_{B_q(\theta R)}u^{2\gamma_i\beta}\ \d\mu\right)^{\frac{1}{2\gamma_i\beta}} = \|u\|_{\infty,\theta R},
    \end{equation}
    where we used 
    \begin{equation*}
        \lim_{i\rightarrow\infty}2\gamma_i\beta = \infty.
    \end{equation*}
    We next estimate the right-hand side of \eqref{will-be-estimated} as
    \begin{eqnarray}
        \prod_{j=0}^{\infty} \left(\gamma_j A + \frac{4}{\sigma_j^2 }\right)^{\frac{1}{2\gamma_j}} &=& \prod_{j=0}^\infty \left( \frac{p\beta^jA}{2} + \frac{4 \cdot4^{j+1}}{(1-\theta)^2R^2}\right)^{\frac{1}{2\gamma_j}}\nonumber\\
        &\leq& \prod_{j=0}^{\infty}\left(\frac{pA}{2} + \frac{16}{(1-\theta)^2R^2}\right)^{\frac{1}{2\gamma_j}}\max\{\beta,4\}^{\frac{j}{2\gamma_j}}\nonumber\\
        &\leq& \left(\frac{pA}{2} + \frac{16}{(1-\theta)^2R^2}\right)^{\frac{\beta}{p(\beta-1)}}\max\{\beta,4\}^{\frac{S_1(\beta)}{p}},\label{sub-mean-eq-2}
    \end{eqnarray}
    where we set $S_1(\beta) = \sum_{j=0}^{\infty} j \beta^{-j} < \infty$.
    Combining \eqref{sub-mean-eq-1} and \eqref{sub-mean-eq-2} with \eqref{will-be-estimated}, we obtain
    \begin{equation*}
    \|u\|_{\infty, \theta R} \leq \left(\frac{R^2}{C_s(R)}\right)^{\frac{\beta}{p(\beta-1)}}\left(\frac{pA}{2} + \frac{16}{(1-\theta)^2R^2}\right)^{\frac{\beta}{p(\beta-1)}}\max\{\beta,4\}^{\frac{S_1(\beta)}{p}}\left( \frac{1}{V_q(R)}\int_{B_q(R)}u^{p}\ \d\mu \right)^{\frac{1}{p}} .
\end{equation*}
Setting 
\begin{equation*}
    C_{11}(p,R) = \left(\frac{1}{C_s(R)}\right)^{\frac{\beta}{p(\beta-1)}}\max\{\beta, 4\}^{\frac{S_1(\beta)}{p}},
\end{equation*}
 we have
\begin{equation}\label{mean-eq-k}
    \|u\|_{\infty, \theta R} \leq C_{11}(p,R) \left(\frac{pAR^2}{2} + \frac{16}{(1-\theta)^2}\right)^{\frac{\beta}{p(\beta-1)}}\left( \frac{1}{V_q(R)}\int_{B_q(R)}u^{p}\ \d\mu \right)^{\frac{1}{p}} .
\end{equation}
This proves the desired inequality for $p \geq 2$. 
In the special case $p =2$, we have
\begin{equation}\label{mean-eq-k=2}
    \|u\|_{\infty, \theta R} \leq C_{11}(2,R) \left(AR^2 + \frac{16}{(1-\theta)^2}\right)^{\frac{\beta}{2(\beta-1)}}\left( \frac{1}{V_q(R)}\int_{B_q(R)}u^{2}\ \d\mu \right)^{\frac{1}{2}} .
\end{equation}

We next consider the case $p < 2$. 
    For any $\theta R \leq \rho \leq R$ and $\eta<1$, \eqref{mean-eq-k=2} yields
    \begin{eqnarray}
        \|u\|_{\infty, \eta \rho} & \leq& C_{11}(2,\rho)\left( A\rho^2 + \frac{16}{(1-\eta)^2} \right)^{\frac{\beta}{2(\beta-1)}} \|u\|_{2,\rho}V_q(\rho)^{-\frac{1}{2}}\nonumber\\
        &\leq &C_{11}(2,R)\left(AR^2 + \frac{16}{(1-\eta)^2}\right)^{\frac{\beta}{2(\beta-1)}}\|u\|_{2, \rho} V_q(\theta R)^{-\frac{1}{2}} \label{will-be-estimated-2}.
    \end{eqnarray}
    Since
    \begin{equation*}
        \left(\int_{B_q(\rho)}u^2\ \d \mu\right)^{\frac{1}{2}}  \leq \left( \int_{B_q(\rho)} u^p \|u\|_{{\infty},\rho}^{2-p} \ \d \mu \right)^{\frac{1}{2}} = \left(\int_{B_q(\rho)} u^p \ \d\mu \right)^{\frac{1}{2}}\|u\|_{{\infty},\rho}^{1-\frac{p}{2}},
    \end{equation*}
    we deduce from \eqref{will-be-estimated-2} that
    \begin{eqnarray}\label{mean-eq-4}
        \|u\|_{\infty, \eta \rho}& \leq& C_{11}(2,R)\left(AR^2 + \frac{16}{(1-\eta)^2}\right)^{\frac{\beta}{2(\beta-1)}}\|u\|_{p, \rho}^{\frac{p}{2}}\|u\|_{\infty, \rho}^{1-\frac{p}{2}} V_q(\theta R)^{-\frac{1}{2}}\label{mean-eq-3}.
    \end{eqnarray}
    Let us choose the sequences $\rho_i$, $\eta_i$ to be
    $$
    \rho_0=\theta R, \quad \rho_i=\theta R+(1-\theta) R \sum_{j=1}^i 2^{-j},\quad \eta_i \rho_i=\rho_{i-1} 
    $$
    for $i\geq 1.$ 
    Substituting $\rho = \rho_i$, $\eta = \eta_i$ to \eqref{mean-eq-4} and iterating this procedure, we obtain
    \begin{eqnarray}
        \|u\|_{\infty,\theta R} &=& \|u\|_{\infty,\rho_0}\nonumber\\
        &=& \|u\|_{\infty,\rho_1\eta_1}\nonumber\\
        &\leq& C_{11}(2,R)\left(AR^2 + \frac{16}{(1-\eta_1)^2}\right)^{\frac{\beta}{2(\beta-1)}}\|u\|_{p,\rho_1}^{\frac{p}{2}}\|u\|_{\infty,\rho_1}^{1-\frac{p}{2}}V_q(\theta R)^{-\frac{1}{2}}\nonumber\\
        &=& C_{11}(2,R)\left(AR^2 + \frac{16}{(1-\eta_1)^2}\right)^{\frac{\beta}{2(\beta-1)}}\|u\|_{p,\rho_1}^{\frac{p}{2}}\|u\|_{\infty,\rho_2 \eta_2}^{1-\frac{p}{2}}V_q(\theta R)^{-\frac{1}{2}}\nonumber\\
        &\leq& C_{11}(2,R)\left(AR^2 + \frac{16}{(1-\eta_1)^2}\right)^{\frac{\beta}{2(\beta-1)}}\|u\|_{p,\rho_1}^{\frac{p}{2}}V_q(\theta R)^{-\frac{1}{2}}\nonumber \\
        &&\qquad\quad \times\left\{ C_{11}(2,R)\left(AR^2 + \frac{16}{(1-\eta_2)^2}\right)^{\frac{\beta}{2(\beta-1)}}\|u\|^{\frac{p}{2}}_{p,\rho_2} V_q(\theta R)^{-\frac{1}{2}}\|u\|^{1-\frac{p}{2}}_{\infty,\rho_2} \right\}^{1-\frac{p}{2}}\nonumber\\
        &\leq & \|u\|_{\infty,\rho_i}^{\left( 1-\frac{p}{2} \right)^i} \prod_{j=1}^i \left\{ C_{11}(2,R)\left(AR^2 + \frac{16}{(1-\eta_j)^2}\right)^{\frac{\beta}{2(\beta-1)}}\|u\|^{\frac{p}{2}}_{p,\rho_j} V_q(\theta R)^{-\frac{1}{2}} \right\}^{\left(1-\frac{p}{2}\right)^{j-1}}\nonumber\\
        &\leq &  \|u\|_{\infty,R}^{\left( 1-\frac{p}{2} \right)^i} \prod_{j=1}^i \left\{ C_{11}(2,R) \left( AR^2 + \frac{16\cdot 2^{2j}}{(1-\theta)^2} \right)^{\frac{\beta}{2(\beta-1)}} \|u\|_{p,R}^{\frac{p}{2}}V_q(\theta R)^{-\frac{1}{2}} \right\}^{\left(1-\frac{p}{2}\right)^{j-1}},\label{iterated-eq}
    \end{eqnarray}
    where we used 
    \begin{equation*}
        \frac{1}{1-\eta_j} = \frac{\rho_j}{\rho_j - \rho_{j-1}} = \frac{\theta + (1-\theta)\sum_{i = 1}^j 2^{-i}}{(1-\theta) 2^{-j}} \leq \frac{2^j}{1-\theta}
    \end{equation*}
    in the last inequality.
    We next estimate the right-hand side of \eqref{iterated-eq}.
    We have
    \begin{eqnarray}
        \prod_{j=1}^{\infty}\left\{\left(AR^2 + \frac{16\cdot 2^{2j}}{(1-\theta)^2}\right)^{\frac{\beta}{2(\beta-1)}}\right\}^{\left(1-\frac{p}{2}\right)^{j-1}} &=& \prod_{j=1}^{\infty}\left[\left\{2^{2j}\left(\frac{AR^2}{2^{2j}} + \frac{16}{(1-\theta)^2}\right)\right\}^{\frac{\beta}{2(\beta-1)}}\right]^{\left(1-\frac{p}{2}\right)^{j-1}}\nonumber\\
        &\leq& \prod_{j=1}^{\infty}\left\{\left(AR^2 + \frac{16}{(1-\theta)^2}\right)^{\frac{\beta}{2(\beta-1)}}2^{j\nu/2}\right\}^{\left(1-\frac{p}{2}\right)^{j-1}}\nonumber\\
        &=& \left(AR^2 + \frac{16}{(1-\theta)^2}\right)^{\frac{\beta}{p(\beta-1)}}2^{\frac{\nu}{2}\sum_{j=1}^{\infty}j\left(1-\frac{p}{2}\right)^{j-1}}\nonumber\\
        &=& \left(AR^2 + \frac{16}{(1-\theta)^2}\right)^{\frac{\beta}{p(\beta-1)}}2^{S_2(p)}\label{estimate-1},
    \end{eqnarray}
    where we set $S_2(p) = \frac{\nu}{2}\sum_{j = 1}^{\infty}j\left(1-\frac{p}{2}\right)^{j-1} < \infty$, and 
    \begin{equation}
        \lim_{i\rightarrow \infty} \|u\|_{\infty, R}^{\left(1-\frac{p}{2}\right)^i} = 1,\quad \lim_{i\rightarrow\infty} \prod_{j=1}^{i}\|u\|_{p, R}^{\frac{p}{2}\left(1-\frac{p}{2}\right)^{j-1}}=\|u\|_{p, R}.\label{estimate-2}
    \end{equation}
    Combining \eqref{estimate-1} and \eqref{estimate-2} with \eqref{iterated-eq}, we have
    $$
    \|u\|_{\infty, \theta R} \leq C_{11}(2,R)^{\frac{2}{p}}\left(AR^2 + \frac{16}{(1-\theta)^2}\right)^{\frac{\beta}{p(\beta-1)}}2^{S_2(p)}\|u\|_{p, R} V_q(\theta R)^{-\frac{1}{p}},
    $$
    which proves the desired inequality for $p < 2$.
    \end{proof}
\section{$L^p$-Liouville theorem}
The purpose of this section is to show the $L^p$-Liouville theorem for $N \in[n,\infty]$. 
In Theorem \ref{lp-liouvolle-eps-range}, we consider the case of $N \in (-\infty,1]\cup[n,\infty]$ and as a special case of Theorem \ref{lp-liouvolle-eps-range}, we show  the $L^p$- Liouville theorem for $N \in [n, \infty]$.
For Theorem \ref{lp-liouvolle-eps-range}, we give two different proofs, one of which does not use the maximum principle while the other proof uses that. 
Our argument provides a proof of Theorem \ref{wu-lp-liouville} different from that in \cite{jy-1} when it is restricted to the case $N = \infty$. 
In particular, our proof does not go through the elliptic Harnack inequality (see also Remark \ref{remark-elliptic-harnack}).


\subsection{The case of $N \in [n,\infty)$}\label{lp-n-section}
We first note that the $L^p$-Liouville theorem for $p > 1$ is obtained without any assumptions on the weighted Ricci curvature as in the following theorem, which is a generalization of Theorem \ref{lp-liouville}.
Although this theorem follows from \cite[Theorem 1.1]{pigola} (we also refer to \cite[Theorem 1.4]{jy-1}), we prove it here for the sake of completeness.
\begin{theorem}\label{lp-liouville-1}
        Let $(M,g,\mu)$ be an $n$-dimensional complete weighted Riemannian manifold. 
        For $p > 1$, let $u$ be a smooth non-negative $L^p(\mu)$-function satisfying $\Delta_\psi u \geq 0$.
         Then $u$ is necessarily constant.
\end{theorem}
\begin{proof}
    We apply the argument in \cite[Theorem 6.3]{yau-1}.
    Let $\lambda = p-1$ and $\varphi$ be a cut-off function such that $0\leq \varphi \leq 1$, $\varphi = 1$ on $B_q(R)$, $\varphi = 0$ on $M \setminus B_q(2R)$ and $|\nabla \varphi|\leq \frac{2}{R}$.
    Since 
    \begin{equation*}
        0 \leq \int_{B_q(2R)} \varphi^2 u^{\lambda} \Delta_\psi u\  \d\mu = - \int_{B_q(2R)} \nabla u\cdot \nabla (\varphi^2 u^{\lambda})\  \d\mu,
    \end{equation*}
    we obtain
    \begin{eqnarray*}
        \lambda \int_{B_q(2R)}u^{\lambda-1}\varphi^2 |\nabla u|^2 \ \d \mu&\leq& 2\int_{B_q(2R)}u^{\lambda}\varphi |\nabla u\cdot \nabla \varphi|\ \d \mu\\
        &\leq& 2\left( \int_{B_q(2R)}|\nabla u|^2 u^{\lambda -1} \varphi^2 \ \d\mu \right)^{\frac{1}{2}}\left(\int_{B_q(2R)}|\nabla \varphi|^2 u^{\lambda + 1} \ \d \mu\right)^{\frac{1}{2}}.
    \end{eqnarray*}
    Therefore, we find
    \begin{eqnarray*}
        \lambda^2 \int_M u^{\lambda-1} \varphi^2 |\nabla u|^2 \ \d \mu \leq 4\int_{B_q(2R)}|\nabla \varphi|^2 u^{1 + \lambda}\ \d\mu\leq  \frac{16}{R^2}\int_M u^p \ \d \mu.
    \end{eqnarray*}
    Letting $R\rightarrow \infty$, we obtain $|\nabla u| = 0$ and the theorem follows.
\end{proof}

Before proving the $L^p$-Liouville theorem, we show the following relative volume comparison theorem. We recall that $V_x(r,R) = \mu(B_x(R)\setminus B_x(r))$.
\begin{theorem}(Relative volume comparison theorem)\label{rel-vol-comparison-thm}
    Let $(M,g,\mu)$ be an $n$-dimensional complete weighted Riemannian manifold and $N \in (-\infty,1] \cup [n,\infty]$,
    $\ez \in \R$ in the $\ez$-range \eqref{epsilin-range}, $K \geq 0$ and $b \ge a>0$.
    Assume that
    \[ \Ric_\psi^N(v)\ge - K\e^{\frac{4(\ez-1)}{n-1}\psi(x)} g(v,v) \]
    holds for all $v \in T_xM \setminus 0$ and
    \begin{equation*}
    a \le \e^{-\frac{2(\ez-1)}{n-1}\psi} \le b.
    \end{equation*}
    Then, we have 
    \begin{equation*}
        V_x(S) \geq \frac{a}{b}\frac{\int_0^{S/b}\bs_{-cK}(\tau)^{1/c} \ \d \tau}{\int_{r/b}^{R/a}\bs_{-cK}(\tau)^{1/c}\ \d \tau} V_x(r,R)
    \end{equation*}
    for arbitrary $R > r \geq 0$ and $S > 0$ satisfying $r \geq S $ and $x\in M$.
\end{theorem}
\begin{proof}
    For an arbitrary unit vector $v\in U_x M $, let $\eta(t) := \exp(t v)$ and $\rho(v) = \sup\{ t > 0 \ |\  d(x,\eta(t)) = t \}$. Setting $\varphi_\eta$ as in \eqref{varphi-def}, we have $\frac{1}{b}\leq \varphi_\eta' \leq \frac{1}{a}$.
    Hence, we obtain
    \begin{equation*}
        \int_{\varphi_\eta(\min\{r,\rho(v)\})}^{\varphi_\eta (\min\{R,\rho(v)\})} h_1(\tau)^{1/c} \ \d \tau = \int_{\min\{r,\rho(v)\}}^{\min\{R,\rho(v)\}}h(t)^{1/c}\varphi_\eta'(t)\  \d t \geq \frac{1}{b}\int_{\min\{r,\rho(v)\}}^{\min\{R,\rho(v)\}} h(t)^{1/c} \ \d t
    \end{equation*}
    and
    \begin{equation*}
        \int_{0}^{\varphi_\eta(\min\{S,\rho(v)\})} h_1(\tau)^{1/c}\  \d \tau \leq \frac{1}{a} \int_{0}^{\min\{S,\rho(v)\}}h(t)^{1/c}\  \d t.
    \end{equation*}
    Since $h_1(\tau)^{\frac{1}{c}}/ \bs_{-cK}(\tau)^{\frac{1}{c}}$ is non-increasing in $\tau$ as mentioned in \eqref{non_increasing_property} and $S\leq R$, we apply the Gromov's lemma (we refer to \cite[Lemma 3.2]{zhu}, for example) and then, we obtain
    \begin{eqnarray*}
        &&\int_{\min\{r,\rho(v)\}}^{\min\{R,\rho(v)\}} h(t)^{1/c}\ \d t \bigg/ \int_{0}^{\min\{S,\rho(v)\}} h(t)^{1/c} \ \d t\\
          &&\qquad\qquad \qquad \leq \frac{b}{a}  \int_{\varphi_\eta(\min\{r,\rho(v)\})}^{\varphi_\eta(\min\{R,\rho(v)\})} h_1(\tau)^{1/c}\ \d \tau \bigg/ \int_{0}^{\varphi_\eta(\min\{S,\rho(v)\})}h_1(\tau)^{1/c}\ \d \tau \\
        &&\qquad \qquad \qquad \leq \frac{b}{a} \int_{\varphi_\eta(\min\{r,\rho(v)\})}^{\varphi_\eta(\min\{R,\rho(v)\})}\bs_{-cK}(\tau)^{1/c}\ \d \tau \bigg/ \int_{0}^{\varphi_\eta(\min\{S,\rho(v)\})}\bs_{-cK}(\tau)^{1/c} \ \d \tau.
    \end{eqnarray*}
    Integrating the above inequality in $v\in U_xM$  with respect to the measure $\Xi$ induced from $g$, we get 
    \begin{eqnarray}
        V_x(r,R) &=& \int_{U_x M}\int_{\min\{r,\rho(v)\}}^{\min\{R,\rho(v)\}} h(t)^{1/c} \ \d t\  \d \Xi( v) \nonumber\\
        &\leq& \frac{b}{a} \int_{U_x M} \left(\int_{\varphi_\eta(\min\{r,\rho(v)\})}^{\varphi_\eta(\min\{R,\rho(v)\})}\bs_{-cK}(\tau)^{1/c}\ \d \tau \bigg/ \int_{0}^{\varphi_\eta(\min\{S,\rho(v)\})}\bs_{-cK}(\tau)^{1/c} \d \tau\right)  \int_0^{\min\{S,\rho(v)\}} h(t)^{1/c}\ \d t \ \d \Xi(v).\nonumber\\
        && \label{rel-vol-eq-1}
    \end{eqnarray}
    Next, we estimate
    \begin{equation*}
         E(v) := \int_{\varphi_\eta(\min\{r,\rho(v)\})}^{\varphi_\eta(\min\{R,\rho(v)\})}\bs_{-cK}(\tau)^{1/c}\ \d \tau \bigg/ \int_{0}^{\varphi_\eta(\min\{S,\rho(v)\})}\bs_{-cK}(\tau)^{1/c} \ \d \tau .
    \end{equation*}
    For $v\in U_xM$ such that $\rho(v) > R$, we have 
    \begin{equation}\label{rel-vol-comparison-eq-1}
        E(v) = \frac{\int_{\varphi_\eta(r)}^{\varphi_\eta(R)} \bs_{-cK}(\tau)^{1/c}\ \d\tau}{\int_{0}^{\varphi_\eta(S)} \bs_{-cK}(\tau)^{1/c}\ \d\tau} \leq \frac{\int_{r/b}^{R/a} \bs_{-cK}(\tau)^{1/c}\ \d\tau}{\int_{0}^{S/b} \bs_{-cK}(\tau)^{1/c}\ \d\tau},
    \end{equation}
    where we used 
    \begin{equation*}
        \frac{r}{b} \leq \varphi_\eta(r)\leq \varphi_\eta (R) \leq \frac{R}{a}.
    \end{equation*}
     For $v\in U_xM$ such that $r \leq \rho(v) \leq R $, we have 
    \begin{equation*}
        E(v) = \frac{\int_{\varphi_\eta(r)}^{\varphi_\eta(\rho(v))} \bs_{-cK}(\tau)^{1/c}\ \d\tau}{\int_{0}^{\varphi_\eta(S)} \bs_{-cK}(\tau)^{1/c}\ \d\tau} \leq \frac{\int_{r/b}^{R/a} \bs_{-cK}(\tau)^{1/c}\ \d\tau}{\int_{0}^{S/b} \bs_{-cK}(\tau)^{1/c}\ \d\tau}.
    \end{equation*}
    In other cases, i.e., for $v\in U_xM$ such that $\rho(v) < r$, we also have 
    \begin{equation*}
        E(v) = 0 \leq \frac{\int_{r/b}^{R/a} \bs_{-cK}(\tau)^{1/c}\ \d\tau}{\int_{0}^{S/b} \bs_{-cK}(\tau)^{1/c}\ \d\tau}.
    \end{equation*}
     Combining these estimates with \eqref{rel-vol-eq-1}, we have
     \begin{eqnarray*}
        V_x(r,R) &\leq& \frac{b}{a} \int_{U_x M} \left(\frac{\int_{r/b}^{R/a} \bs_{-cK}(\tau)^{1/c}\ \d\tau}{\int_{0}^{S/b} \bs_{-cK}(\tau)^{1/c}\ \d\tau}\right) \int_0^{\min\{S,\rho(v)\}} h(t)^{1/c}\ \d t \ \d \Xi(v)\nonumber\\
        &=& \frac{b}{a} \frac{\int_{r/b}^{R/a} \bs_{-cK}(\tau)^{1/c}\ \d\tau}{\int_{0}^{S/b} \bs_{-cK}(\tau)^{1/c}\ \d\tau}V_x(S).
     \end{eqnarray*}
     Hence, we obtain the desired inequality.
\end{proof}
\begin{remark}\label{lower-bound-d-remark}\rm{
    Although Theorem \ref{rel-vol-comparison-thm} only considered the case $S\leq r$, we consider the case $r < S \leq R$ in this remark for the proof of the next theorem.
    For $v\in U_x M$ such that $r < \rho(v) < S$, we can also estimate $E(v)$ from above. Indeed, we have
    \begin{equation*}
        E(v) \leq \frac{\int_{\varphi_\eta(r)}^{\varphi_\eta(\rho(v))} \bs_{-cK}(\tau)^{1/c}\ \d\tau}{\int_{0}^{\varphi_\eta(\rho(v))} \bs_{-cK}(\tau)^{1/c}\ \d\tau} \leq 1.
    \end{equation*}}
\end{remark}

Using the relative volume comparison in Theorem \ref{rel-vol-comparison-thm}, we show the following theorem,
which will be used to prove the $L^p$-Liouville theorem for $N \in [n,\infty]$.
\begin{theorem}\label{lp-liouvolle-eps-range}
    Let $(M,g,\mu)$ be an $n$-dimensional complete weighted Riemannian manifold and $N \in (-\infty,1] \cup [n,\infty]$,
    $\ez \in \R$ in the $\ez$-range \eqref{epsilin-range} and $b \ge a>0$. We assume 
    \begin{equation*}
    a \le \e^{-\frac{2(\ez-1)}{n-1}\psi} \le b
    \end{equation*}
    and 
    \begin{equation}\label{odd-assumption}
        \frac{b}{a}\left\{\left(\frac{b}{a}\right)^{1 + \frac{1}{c}} - 1\right\} < \frac{1}{20^{1 + \frac{1}{c}}}
    \end{equation}
    For $p\in (0,\infty)$,  let $u$ be a smooth non-negative $L^p(\mu)$ function satisfying $\Delta_\psi u \geq 0$.
    Then, there exists a constant $\delta > 0$ depending on $c,a,b$ such that the following property holds:

    If, for some $q\in M$, we have
    \begin{equation}\label{curvature-assume-lp}
        \Ric_\psi^N \ge -\delta \e^{\frac{4(\ez-1)}{n-1}\psi(x)}d(q,x)^{-2} 
    \end{equation}
    when $d(q,x)$ is sufficiently large, then $u$ is identically zero.
\end{theorem}
\begin{remark}\rm{
    The assumption \eqref{odd-assumption} is used only to obtain \eqref{lp-lio-eq-3} in the following proof.
    It seems difficult to the author to prove \eqref{lp-lio-eq-3} without \eqref{odd-assumption},
    and this is the only reason why we assume \eqref{odd-assumption}.
    In an attempt trying to remove \eqref{odd-assumption}, the author considered using the volume comparison theorems for reparametrized distance in \cite{yoroshikin,kuwae_sakurai} to obtain \eqref{lp-lio-eq-3}, which did not work since the reparametrized distance does not satisfy the triangle inequality.
    Although in the case $N\in[n,\infty)$, $\eps = 1$ and $a = b = 1$, the assumption \eqref{odd-assumption} is satisfied, it is not always satisfied especially in the case $N \in (-\infty,1]\cup \{\infty\}$
    }.
\end{remark}
\begin{proof}
    We apply the arguments in \cite[Theorem 2.5]{li-schoen} and \cite[Theorem 6.1]{jy-1}.
    We first show that $u$ is necessarily constant.
    Since the case $p > 1$ is obtained in Theorem \ref{lp-liouville-1},
    we only consider the case $0 < p \leq 1$.
    If $u(x)$ goes to $0$ when $x$ is far enough away from $q$, then $u$ is an $L^{\infty}(\mu)$-function.
    In such a case, from $u\in L^p(\mu) \cap L^{\infty}(\mu)$, we deduce $u\in L^2(\mu)$ and, from Theorem \ref{lp-liouville-1}, $u$ is necessarily constant.
    In the following argument, we show that $u(x)$ is close to $0$ when $x$ is sufficiently away from $q$. 
    
    Let $x\in M$ and $\gamma:[0,T]\rightarrow M$ be a minimal geodesic satisfying $\gamma(0) = q, \gamma(T) = x$, where $T = d(q,x)$. For fixed $\beta > 1$, we set 
    \begin{equation*}
        t_0 = 0,\quad t_1 = 1 + \beta, \quad t_i = 2\sum_{j = 0}^i \beta^j - 1 -\beta^i
    \end{equation*}
    and we take $k\in \mathbb{N}$ such that $t_k \leq T$ and $t_{k + 1} > T$. For this $\{t_i\}_{1\leq i \leq k}$, we set $x_i = \gamma(t_i)$ and observe
    \begin{equation}\label{x-k-const}
        d(x_i, x_{i + 1}) = \beta^i + \beta^{i + 1}, \quad d(q,x_i) = t_i, \quad d(x_k,x) < \beta^k + \beta^{k + 1}.
    \end{equation}
    For $i\leq k$, by the relative volume comparison in Theorem \ref{rel-vol-comparison-thm}, we have
    \begin{equation}\label{lp-lio-eq-15}
        V_{x_i}\left(\frac{\beta^i}{20}\right) \geq D_i V_{x_i}(\beta^i,\beta^i + 2\beta^{i -1}) \geq D_i V_{x_{i-1}}\left(\frac{\beta^{i-1}}{20}\right),
    \end{equation}
    where
    \begin{equation*}
        D_i = \frac{a}{b}\left(\int_0^{\beta^i/(20b)}\bs_{-cK}(\tau)^{1/c}\ \d\tau \right)\bigg/\left( \int^{(\beta^i + 2\beta^{i-1})/a}_{\beta^i/b}\bs_{-cK}(\tau)^{1/c} \ \d\tau\right)
    \end{equation*}
    with $K = K_{\eps}(x_i,\beta^i + 2\beta^{i-1})$.
    Iterating this inequality, we have 
    \begin{equation}\label{lp-lio-eq-2}
        V_{x_k}\left(\frac{\beta^k}{20}\right) \geq \left(\prod_{i = 1}^k D_i \right)V_q\left(\frac{1}{20}\right).
    \end{equation}
    We show $V_{x_k}\left(\beta^k / 20\right)\rightarrow \infty$ as $d(q,x)\rightarrow \infty$. 
    We first observe that we can make $\beta^i\sqrt{K_\eps(x_i, \beta^i + 2\beta^{i-1})}$ arbitrarily small by taking $\delta$ sufficiently small and we next use this fact to approximate the value of $D_i$.
    For $y\in B_{x_i}(\beta^i + 2\beta^{i-1})$, we have 
    \begin{eqnarray}
        d(q,y) &\geq& d(q,x_i) - d(x_i,y)\nonumber\\
        &\geq& \left(2\sum_{j = 0}^{i}\beta^j - 1 - \beta^i \right)- (\beta^i + 2\beta^{i-1})\nonumber\\
        &=& 2\sum_{j = 0}^{i-2}\beta^j - 1\nonumber\\
        &=& \frac{1-2\beta^{i-2} + \beta}{1- \beta}\label{lp-lio-eq-1}.
    \end{eqnarray}
    Combining this with the assumption of the curvature \eqref{curvature-assume-lp} (we take $\delta$ sufficiently small later), for sufficiently large $i$, we have
    \begin{eqnarray*}
        \beta^i\sqrt{K_\eps\left( x_i,\beta^i + 2\beta^{i-1} \right)}& = & \beta^i\sqrt{\sup_{y\in B_i}K_\eps(y)}\\
        &\leq&\beta^i\sup_{y\in B_i}\frac{\sqrt{\delta}}{d(q,y)}\\
        &\leq & \frac{\beta^2(\beta-1)\sqrt{\delta}}{2-\beta^{2-i} - \beta^{3-i}},
    \end{eqnarray*}
    where we denote $B_i = B_{x_i}\left(\beta^i + 2\beta^{i-1}\right)$.
    Therefore, $\beta^i\sqrt{K_\eps(x_i,\beta^i + 2\beta^{i-1})}$ can be made arbitrarily small by taking $\delta$ small enough for a fixed $\beta$, which allows the following approximation of $D_i$.
    By the first order approximation of $\sinh$, $D_i$ is approximated:
    \begin{eqnarray}
        \frac{a}{b}\frac{\{\beta^i\sqrt{K}/(20b)\}^{1 + \frac{1}{c}}}{((\beta^i + 2\beta^{i-1})\sqrt{K}/a)^{1 + \frac{1}{c}} - (\beta^i \sqrt{K}/ b)^{1 + \frac{1}{c}}}&=& \frac{a}{b}\frac{1}{20^{1 + \frac{1}{c}}} \frac{1}{\left\{\left( 1 + 2/\beta\right)b/a\right\}^{1 + 1/c} - 1}\label{d-i-approximation-1},
    \end{eqnarray}
    where we denote $K = K_\eps(x_i,\beta^i + 2\beta^{i-1})$.
    From the assumption \eqref{odd-assumption}, we can make \eqref{d-i-approximation-1} larger than $1$ by taking $\beta$ sufficiently large depending on $c,a,b$,  and then taking $\delta$ sufficiently small so that  the left-hand side of \eqref{d-i-approximation-1} approximate $D_i$,
    we can make $D_i$ larger than $1$. Hence, from \eqref{lp-lio-eq-2}, we deduce
    \begin{equation}\label{lp-lio-eq-3}
        \lim_{k\rightarrow \infty}V_{x_k}\left(\frac{\beta^k}{20}\right) = \infty.
    \end{equation}


    In the following Case 1 and Case 2, we estimate $u(x)$ when $x$ is far enough away from $q$.
    
    \underline{Case 1: When $d(x,x_k) <  \beta^k / 20$.}
    
    By the mean value inequality \eqref{mean-value-inequality-k}, we have
    \begin{eqnarray}
        u(x)&\leq& \sup_{B_{x_k}(\beta^k/20)}u\leq \widetilde{E}_4 \exp\left(\widetilde{D}_4\sqrt{K_\eps(x_k,\beta^k)}\beta^k \right) V_{x_k}\left(\frac{\beta^k}{20}\right)^{-\frac{1}{p}} \|u\|_p\label{lp-lio-eq-4},
    \end{eqnarray}
    where $\widetilde{D}_4$ and $\widetilde{E}_4$ are constants depending on $p,c,a,b,n$.
    

    \underline{Case 2: When $d(x,x_k) \geq \beta^k/20$.}
    
    By the mean value inequality \eqref{mean-value-inequality-k}, we also get 
    \begin{eqnarray}
        u(x) &\leq& \sup_{B_x(\beta^k/20)}u\leq \widetilde{E}_4 \exp\left( \widetilde{D}_4\sqrt{K_{\eps}(x,\beta^k)} \beta^k \right) V_x\left(\frac{\beta^k}{20}\right)^{-\frac{1}{p}}\|u\|_p \label{lp-lio-eq-5}.
    \end{eqnarray}
    We estimate the right-hand side of \eqref{lp-lio-eq-5}.
    Note that
    \begin{equation*}
        B_{x_k}\left( \frac{\beta^k}{20} \right) \subset B_{x}\left(d(x,x_k)-\frac{\beta^k}{20}, d(x,x_k) + \frac{\beta^k}{20} \right).
    \end{equation*}
Combining this with the argument in the proof of Theorem \ref{rel-vol-comparison-thm}, we have
    \begin{eqnarray}\label{lp-lio-eq-6}
        V_{x_k}\left(\frac{\beta^k}{20}\right) &\leq& V_x\left(d(x,x_k) - \frac{\beta^k}{20},d(x,x_k) + \frac{\beta^k}{20}\right) \leq \frac{b}{a} \int_{U_x M} E(v) \int_0^{\min\left\{\frac{\beta^k}{20}, \rho(v)\right\}}h(t)^{1/c}\ \d t\ \d \Xi(v),\qquad
    \end{eqnarray}
    where 
    \begin{equation}\label{e-est-1}
        E(v) =  \int_{\varphi_\eta\left(\min\left\{d(x,x_k) - \frac{\beta^k}{20},\rho(v)\right\}\right)}^{\varphi_\eta\left(\min\left\{d(x,x_k) + \frac{\beta^k}{20},\rho(v)\right\}\right)}\bs_{-cK}(\tau)^{1/c}\ \d \tau   \bigg/          \int_{0}^{\varphi_\eta\left(\min\left\{\frac{\beta^k}{20},\rho(v)\right\}\right)}\bs_{-cK}(\tau)^{1/c} \ \d \tau
    \end{equation}
    with $K = K_{\eps}\left(x,d(x,x_k) + \frac{\beta^k}{20}\right)$.
    In the case where $v\in U_xM$ satisfies $d(x,x_k) - \frac{\beta^k}{20} < \rho(v) <  \frac{\beta^k}{20}$, 
     we have
    \begin{equation}\label{lp-lio-eq-11}
        E(v) \leq 1
    \end{equation}
    from Remark \ref{lower-bound-d-remark}.
    In other cases, from the argument in Theorem \ref{rel-vol-comparison-thm}, we have 
    \begin{equation}\label{lp-lio-eq-12}
        E(v) \leq    \int_{\left(d(x,x_k)-\frac{\beta^k}{20}\right)/b}^{\left(d(x,x_k)+\frac{\beta^k}{20}\right) / a}\bs_{-cK}(\tau)^{1/c}\ \d \tau \bigg/  \int_0^{\beta^k/(20b)}\bs_{-cK}(\tau)^{1/c}\ \d \tau.
    \end{equation}
    We shall show that the right-hand side of \eqref{lp-lio-eq-12} is bounded from above.
    Since we have
    \begin{eqnarray*}
        \sqrt{K_\eps\left(x,d(x,x_k) + \beta^k\right)} &=& \sqrt{\sup_{y\in B_k}K_\eps(y)}\\
        &\leq& \sup_{y\in B_k}\frac{\sqrt{\delta}}{d(q,y)}\\
        &\leq& \frac{\sqrt{\delta}}{1 + 2\beta + \cdots + 2\beta^{k-1}}\\
        &=& \frac{(1-\beta)\sqrt{\delta}}{1-2\beta^k + \beta},
    \end{eqnarray*}
    where we denote $B_k = B_x(d(x,x_k) + \beta^k)$, we get
    \begin{eqnarray}
        \left(d(x,x_k) + \frac{\beta^k}{20}\right)\sqrt{K_\eps\left(x,d(x,x_k) + \frac{\beta^k}{20} \right)}  &\leq &(d(x,x_k) + \beta^k)\sqrt{K_\eps(x,d(x,x_k) + \beta^k)}\nonumber \\
        &\leq& (\beta^{k + 1} + 2\beta^k)\frac{1-\beta}{1 + \beta - 2\beta^{k}}\sqrt{\delta}\nonumber\\
        &=& (2 +\beta)\frac{(\beta-1)}{2- \beta^{1-k} - \beta^{-k}}\sqrt{\delta}\nonumber.
    \end{eqnarray}
    Hence, if $k$ is sufficiently large, we obtain
    \begin{equation}\label{lp-liouville-eq-8}
        \left(d(x,x_k) + \frac{\beta^k}{20}\right)\sqrt{K_\eps\left(x,d(x,x_k) + \frac{\beta^k}{20}\right)} \leq (2 + \beta)(\beta-1) \sqrt{\delta}.
    \end{equation}
    By taking $\delta$ small enough for some fixed $\beta$, the right-hand side of \eqref{lp-liouville-eq-8} can be made arbitrarily small.
    Then, from the first order approximation, $E(v)$ is approximated by the left-hand side of \eqref{e-approximation-1} and it is bounded from above by \eqref{d-approximation-1} as follows:
    \begin{eqnarray}
        \frac{\left(\left(d(x,x_k) + \frac{\beta^k}{20}\right)/a\right)^{1 + \frac{1}{c}} - \left(\left(d(x,x_k) - \frac{\beta^k}{20}\right)/b\right)^{1 + \frac{1}{c}}}{(\beta^k/(20b))^{1 + \frac{1}{c}}}&\leq &  \frac{\left(\left(d(x,x_k) + \beta^k\right)/a\right)^{1 + \frac{1}{c}}}{(\beta^k/(20b))^{1 + \frac{1}{c}}} \label{e-approximation-1}\\
        &\leq &  \frac{\left(\left(2\beta^k + \beta^{k + 1}\right)/a\right)^{1 + \frac{1}{c}}}{(\beta^k/(20b))^{1 + \frac{1}{c}}} \nonumber\\
        &=& 20^{1 + \frac{1}{c}} \frac{b^{1 + \frac{1}{c}}}{a^{1 + \frac{1}{c}}}\left(2 + \beta\right)^{1 + \frac{1}{c}}\label{d-approximation-1}.
    \end{eqnarray} 
    Combining \eqref{lp-lio-eq-6} with \eqref{lp-lio-eq-11} and \eqref{d-approximation-1}, 
    we see that there exists $C_{12}$ depending on $c,a,b,\beta$ such that 
    \begin{equation*}
        V_{x_k}\left(\frac{\beta^k}{20}\right) \leq  V_x\left( d(x,x_k) - \frac{\beta^k}{20}, d(x,x_k) + \frac{\beta^k}{20}\right) \leq C_{12} V_x\left(\frac{\beta^k}{20}\right).
    \end{equation*}
    Combining this with \eqref{lp-lio-eq-5}, we obtain 
    \begin{equation*}
        u(x)\leq \widetilde{E}_5 \exp\left(\widetilde{D}_5\sqrt{K_\eps(x,\beta^k)}\beta^k \right)V_{x_k}\left(\frac{\beta^k}{20}\right)^{-\frac{1}{p}} \|u\|_p ,
    \end{equation*}
    where $\widetilde{D}_5$ and $\widetilde{E}_5$ are constatnts depending on $p,c,a,b,n,\beta$.
    
     
     By Case 1 and Case 2, we can bound $u(x)$ from above as
     \begin{equation}\label{lp-lio-eq-10}
        u(x)\leq \widetilde{E}_6\exp\left\{ \widetilde{D}_6\max\left( \sqrt{K_\eps\left(x_k,\beta^k\right)}, \sqrt{K_\eps(x,\beta^k)}\right)\beta^k \right\} V_{x_k}\left(\frac{\beta^k}{20}\right)^{-\frac{1}{p}}\|u\|_p,
     \end{equation}
     where $\widetilde{D}_6$ and $\widetilde{E}_6$ are constants depending on $p,c,a,b,n,\beta$. 
     Since $k$ increases as $x$ moves away from $q$, we deduce that $V_{x_k}(\beta^k/ 20) \rightarrow \infty $ as $d(q,x)\rightarrow \infty$ by \eqref{lp-lio-eq-3}.
     Since $\beta^k\sqrt{K_{\eps}(x_k,\beta^k)}$ and $\beta^k \sqrt{K_\eps(x,\beta^k)}$ are bounded from above as $d(q,x)\rightarrow \infty$ by the assumption \eqref{curvature-assume-lp}, the right-hand side of \eqref{lp-lio-eq-10} goes to $0$ as $d(q,x)\rightarrow \infty$.
     Therefore, recalling the argument at the beginning of this proof, $u$ is a constant function.
     In the argument above, we take $\beta$ so as to make \eqref{d-i-approximation-1} larger than $1$, and for this $\beta$, we take $\delta$ so that the approximations of $D_i$ in \eqref{d-i-approximation-1} and $E(v)$ in \eqref{e-approximation-1} hold. 
     Therefore, we see that $\delta$ depends only on $c,a,b$. 
     
    From \eqref{lp-lio-eq-3}, we have $\mu(M) = \infty$ for sufficiently small $\delta$. Hence, every constant $L^2$-function is identically zero and the theorem follows. 
\end{proof}
Theorem \ref{lp-liouvolle-eps-range} can also be shown without going through the argument in Case 2.
Indeed, we have the following another proof of Theorem \ref{lp-liouvolle-eps-range} which uses the maximum principle.

\underline{\textit{Alternative proof of Theorem \ref{lp-liouvolle-eps-range}}}

In the previous proof, the value of $u(x)$ was evaluated by Case 1 and Case 2 to show that $u(x)$ approaches $0$ as $x$ moves away from $q$. 
On the other hand, in this proof, we combine the argument of Case 1 with the maximum principle to show that $u(x)$ goes down to $0$ when $x$ moves away from $q$.
We take $\beta$ sufficiently large so as to make \eqref{d-i-approximation-1} larger than $1$, and for this $\beta$, we take $\delta$ so that the approximation \eqref{d-i-approximation-1} of $D_i$ holds.
For $j\geq 2$, we set 
\begin{equation*}
    \mathcal{F}_{1,j} = \left\{y\in M\ \bigg|\ 1 + (2\beta + \cdots + 2\beta^{j-1}) + \beta^j < d(q,y) < 1 + (2\beta + \cdots + 2\beta^{j-1}) + \beta^j + \frac{\beta^j}{20}\right\}
\end{equation*}
and 
\begin{equation*}
    \mathcal{F}_{2,j} = \left\{y\in M\ \bigg|\ 1 + (2\beta + \cdots + 2\beta^{j-1}) + \beta^j + \frac{\beta^j}{20} \leq d(q,y)\leq 1 + (2\beta + \cdots + 2\beta^{j-1} + 2\beta^j) + \beta^{j + 1} \right\}.
\end{equation*}
By the maximum principle, $u$ takes its maximum value on $\mathcal{F}_{2,j}$ at the boundary of $\mathcal{F}_{2,j}$, which we prove in the following argument for the sake of completeness.

By Theorem \ref{laplacian_comparison_theorem}, for $x\in \mathcal{F}_{2,j}$ and $d_q(x) = d(q,x)$, we have
\begin{equation}\label{another-proof-lp-eq-1}
    \Delta_\psi d_q(x) \leq \frac{\sqrt{K}}{a\sqrt{c}}\coth\left(\frac{\sqrt{cK}}{b}d_q(x)\right) \leq \frac{\sqrt{K}}{a\sqrt{c}} + \frac{b}{acd_q(x)},
\end{equation}
where we set
\begin{equation*}
    K = \sup_{y\in B}K_\eps(y) \quad \mbox{for} \quad B = \{y\in M \ |\  d(q,y) \leq 1 + (2\beta + \cdots + 2\beta^{j-1} + 2\beta^j) + \beta^{ j+ 1}\}.
\end{equation*}
For $x\in \mathcal{F}_{2,j}^{\circ}$ ($A^{\circ}$ denotes the interior of $A$),
by the definition of $\mathcal{F}_{2,j}$, we have 
\begin{equation*}
    \Delta_\psi d_q(x) \leq \frac{\sqrt{K}}{a\sqrt{c}}+ \frac{b}{ac}\frac{1}{1 + (2\beta + \cdots + 2\beta^{j-1}) + \beta^j + \frac{\beta^j}{20} } =: G.
\end{equation*}
 Hence, for $x\in \mathcal{F}_{2,j}^{\circ}$, we have
\begin{eqnarray}
    \Delta_\psi \exp\left(-2Gd_q(x)\right) &=& (-2G\Delta_\psi d_q(x) + 4G^2)\exp\left(-2Gd_q(x)\right) \nonumber\\
    &\geq& 2G^2 \exp\left(-2G d_q(x)\right) > 0.\label{another-proof-eq-4}
\end{eqnarray}
Here, for $\eps > 0$, we set 
\begin{equation*}
    u_\eps = u + \eps \e^{-2G d_q(x)}.
\end{equation*}
From \eqref{another-proof-eq-4} and $\Delta_\psi u \geq 0$, we see that $\Delta_\psi u_\eps > 0$ in the interior of $\mathcal{F}_{2,j}$. If $u_\eps$ takes its maximum at $z\in (\mathcal{F}_{2,j})^{\circ}$, we obtain 
\begin{equation*}
    \nabla u_\eps (z) = 0, \quad \Delta_\psi u_\eps(z) = \Delta u_\eps(z) \leq 0.
\end{equation*}
Since this leads to the contradiction, $u_\eps$ takes its maximum value at the boundary. 
Letting $\eps\rightarrow 0$, we see that $u$ also takes its maximum value on $\mathcal{F}_{2,j}$ at the boundary of $\mathcal{F}_{2,j}$. 

Next, we estimate $u$ on $\mathcal{F}_{1,k}$. 
For $x\in \mathcal{F}_{1,k}$, we construct the sequence $\{x_i\}_{1\leq i \leq k}$ as in \eqref{x-k-const} for $\beta > 0$. From \eqref{lp-lio-eq-2} and \eqref{lp-lio-eq-4}, we have 
\begin{equation*}
    u(x) \leq  \widetilde{E}_4 \exp\left( \widetilde{D}_4\sqrt{K_\eps(x_k,\beta^k)}\beta^k \right) \left\{\left(\prod_{i = 1}^k D_i \right)V_q\left(\frac{1}{20}\right)\right\}^{-\frac{1}{p}} \|u\|_p,
\end{equation*}
 where $D_i, \widetilde{D}_4, \widetilde{E}_4$ are taken as in the proof above.
From the assumption \eqref{curvature-assume-lp}, we have 
\begin{eqnarray*}
    \beta^k\sqrt{K_\eps(x_k,\beta^k)} &\leq&  \frac{\beta^k \sqrt{\delta}}{1 + 2\beta + \cdots + 2\beta^{k-1}} \leq \frac{\beta\sqrt{\delta}}{2}.
\end{eqnarray*}
Hence, we have 
\begin{equation}\label{another-proof-eq-2}
    u(x) \leq \widetilde{E}_4 \exp\left( \frac{\beta \widetilde{D}_4 \sqrt{\delta}}{2}\right) \left\{\left(\prod_{i = 1}^k D_i \right)V_q\left(\frac{1}{20}\right)\right\}^{-\frac{1}{p}} \|u\|_p.
\end{equation}
Since \eqref{another-proof-eq-2} holds for each $k\geq 2$, we have 
\begin{equation}\label{another-proof-eq-3}
    u(x) \leq \widetilde{E}_4 \exp\left( \frac{\beta \widetilde{D}_4 \sqrt{\delta}}{2}\right) \left\{\left(\prod_{i = 1}^{k+1} D_i \right)V_q\left(\frac{1}{20}\right)\right\}^{-\frac{1}{p}} \|u\|_p
\end{equation}
for $x\in \mathcal{F}_{1,k + 1}$.
Recalling that $\mathcal{F}_{2,k}$ is sandwiched with $\mathcal{F}_{1,k}$ and $\mathcal{F}_{1,k+1}$, 
we see that the value of $u$ at the boundary of $\mathcal{F}_{2,k}$ is bounded from above by \eqref{another-proof-eq-2} and \eqref{another-proof-eq-3} by the maximum principle.
Hence, we obtain
\begin{equation*}
    u(x) \leq \widetilde{E}_4 \exp\left( \frac{\beta \widetilde{D}_4 \sqrt{\delta}}{2}\right) \left\{\left(\prod_{i = 1}^k D_i \right)V_q\left(\frac{1}{20}\right)\right\}^{-\frac{1}{p}} \|u\|_p
\end{equation*}
for sufficiently large $k$ and $x\in\mathcal{F}_{2,k}$. 
Therefore, we see that $u(x)$ goes to $0$ when $x$ moves away from $q$ and the theorem follows. \qed


 Considering the case of $N \in [n,\infty)$, $a = b = 1$ and $\eps = 1$ in Theorem \ref{lp-liouvolle-eps-range}, we obtain the following $L^p$-Liouville theorem.

\begin{theorem}($L^p$-Liouville theorem for $N\in [n,\infty)$)\label{lp-n-liouville}
    Let $(M,g,\mu)$ be an $n$-dimensional complete weighted Riemannian manifold and $N \in  [n,\infty)$.
    For $p\in (0,\infty)$,  let $u$ be a smooth non-negative $L^p(\mu)$ function satisfying $\Delta_\psi u \geq 0$.
    Then, there exists a constant $\delta > 0$ depending on $c,a,b$ such the following property holds:

    If, for some $q\in M$, we have
    \begin{equation}
        \Ric_\psi^N \ge -\delta d(q,x)^{-2} 
    \end{equation}
    when $d(q,x)$ is sufficiently large, then $u$ is identically zero.
\end{theorem}

\subsection{The case of  $N = \infty$}\label{lp-infty-section}

In this subsection, by combining the arguments of Theorem \ref{lp-liouvolle-eps-range} and the following Theorem \ref{infty-relative-volume-comparison-thm}, we prove Theorem \ref{wu-lp-liouville}.
Our proof is essentially different from that of \cite{jy-1} as mentioned in the introduction of Section 4.
Although the proof of Theorem \ref{lp-liouvolle-eps-range} used \eqref{lp-lio-eq-3}, this was not guaranteed without the assumption of \eqref{odd-assumption}. 
However, in the case of $N = \infty$, using the following volume comparison in Theorem \ref{infty-relative-volume-comparison-thm} instead of that in Theorem \ref{rel-vol-comparison-thm}, we can obtain \eqref{lp-lio-eq-3} without assuming \eqref{odd-assumption}.
This is the same strategy taken in \cite[Theorem 6.1]{jy-1}.
Although the following theorem is obtained in \cite[(4.10)]{wei-wylie}, we prove it here for the sake of completeness.
\begin{theorem}\label{infty-relative-volume-comparison-thm}
    Let $(M,g,\mu)$ be an $n$-dimensional complete Riemannian manifold. We assume $\Ric_\psi^{\infty}\geq -K$ with $K > 0$ and $|\psi| \leq A$ for some constant $A > 0$.
    Then, we have 
    \begin{equation}\label{infty-relative-volume-comparison}
        V_x(S) \geq \frac{\int_0^S \bs_{-K/(n-1)} (t)^{n-1 + 4A} \ \d t}{\int_r^R \bs_{-K/(n-1)} (t)^{n-1 + 4A} \ \d t}V_x(r,R)
    \end{equation}
    for arbitrary $R > r \geq 0$ and $S > 0$ satisfying $r\geq S$ and $x\in M$.
\end{theorem}
\begin{proof}
    For $x,y\in M$, let $\gamma$ be a unit speed minimizing geodesic such that $\gamma(0) = x$ and $\gamma(s) = d(x,y)$, where $s = d(x,y)$. 
    For the unit tangent vector $\xi := \left.\partial_s \gamma(s)\right|_{s=0}$, let $I_\psi(x, s, \xi)$ be the Jacobian of the map $\exp_x : T_xM\rightarrow M$ at $s\xi$ with respect to $\mu$. Then
$$
\d\mu=I_\psi(x, s, \xi) \ \d s \ \d \xi,
$$
where $\d \xi$ is the usual measure on the sphere. 
For $I_\psi$, we have the following inequality in \cite[(3.15)]{wei-wylie}:
    \begin{equation}\label{infty-laplacian-comparison}
        \frac{(I_\psi(x,s,\xi))'}{I_\psi(x,s,\xi)} \leq \sqrt{(n-1)K}\coth \left(\sqrt{\frac{K}{n-1}} s\right) + \frac{1}{\bs_{-K/(n-1)}(s)^2}\int_0^s \{\psi(\gamma(t)) - \psi(\gamma(s))\}(\bs_{-K/(n-1)}(t)^2 )''\ \d t.
    \end{equation}
    Combining this with the assumption $|\psi| \leq A$, we have 
    \begin{equation*}
        \frac{(I_\psi(x,s,\xi))'}{I_\psi(x,s,\xi)} \leq (n-1 + 4A)\frac{(\bs_{-K/(n-1)}(s))'}{\bs_{-K/(n-1)}(s)}.
    \end{equation*} 
    Therefore, we see that $I_\psi(x,s,\xi) / \bs_{-K/(n-1)}(s)^{n-1 + 4A}$ is non-increasing in $s$. Applying this fact instead of \eqref{non_increasing_property} to the argument in Theorem \ref{rel-vol-comparison-thm}, we obtain the desired inequality \eqref{infty-relative-volume-comparison}.
\end{proof}

In the following proof of Theorem \ref{wu-lp-liouville}, we obtain \eqref{lp-lio-eq-3} using Theorem \ref{infty-relative-volume-comparison-thm} and, after that, by directly applying the argument in the proof of Theorem \ref{lp-liouvolle-eps-range}, we obtain Theorem \ref{wu-lp-liouville}. 
Hence, two proofs of Theorem \ref{lp-liouvolle-eps-range} also give two different ways of proving Theorem \ref{wu-lp-liouville} when it is restricted to the case $N = \infty$. 

\underline{\textit{Alternative proof of Theorem \ref{wu-lp-liouville}}}

We denote local bounds of $\Ric^\infty_{\psi}$ as 
\begin{equation*}
    \widetilde{K}_\infty (x) = \max\left\{ 0, \sup_{v\in U_xM}\left( -\Ric_\psi^{\infty}(v) \right) \right\} \quad \mbox{and}\quad \widetilde{K}_\infty(q,R) = \sup_{y\in B_q(R)}\widetilde{K}_\infty(y).
\end{equation*}
For $x\in M$ and $\beta > 0$,
we construct the sequence $\{x_i\}_{1\leq i\leq k}$ as in \eqref{x-k-const}.
Applying Theorem \ref{infty-relative-volume-comparison-thm} in the same way as we applied Theorem \ref{rel-vol-comparison-thm} in \eqref{lp-lio-eq-15}, we have 
\begin{equation*}
    V_{x_i}\left(\frac{\beta^i}{20}\right) \geq D_i V_{x_i}(\beta^i,\beta^i + 2\beta^{i-1}) \geq D_i V_{x_{i-1}}\left(\frac{\beta^{i-1}}{20}\right)
\end{equation*}
for $i\geq 1$, where
\begin{equation*}
    D_i = \int_0^{\beta^i/ 20} \bs_{-K/(n-1)}(t)^{n - 1 + 4A}\ \d t\bigg/ \int_{\beta^i}^{\beta^i + 2\beta^{i-1}} \bs_{-K/(n-1)}(t)^{n-1 + 4A}\ \d t,
\end{equation*}
with $K = \widetilde{K}_\infty(x_i,\beta^i + 2\beta^{i-1})$. 
Hence, $D_i$ is approximated by
\begin{equation}\label{d-i-approximation-infty}
    \frac{(\beta^i/20)^{n + 4A}}{(\beta^i + 2\beta^{i-1})^{n + 4A} - (\beta^i)^{n + 4A}} = \frac{(1/20)^{n + 4A}}{(1 + 2/\beta)^{n + 4A} - 1}
\end{equation}
for sufficiently small $\delta$. By taking $\beta$ so that
\begin{equation}\label{beta-dependance}
    \beta > \frac{2}{\left(20^{-(n + 4A)} + 1\right)^{1/(n + 4A)} - 1} > 1
\end{equation}
holds, we can make \eqref{d-i-approximation-infty} larger than $1$. Therefore, by taking $\delta$ sufficiently small so that the approximation \eqref{d-i-approximation-infty} holds, we can make $D_i$ larger than $1$ and we obtain \eqref{lp-lio-eq-3}.
 
The argument after \eqref{lp-lio-eq-3} in the proof of Theorem \ref{lp-liouvolle-eps-range} is applicable to the situation of Theorem \ref{wu-lp-liouville}.
Indeed, under the assumption of Theorem \ref{wu-lp-liouville},
    we fix $\eps\in (-1,1)$, which is the $\eps$-range for $ N = \infty$, and set $a = \exp\left(-\frac{2(1-\eps)}{n-1}A\right)$ and $b = \exp\left(\frac{2(1-\eps)}{n-1}A\right)$. Since we have the relation
    \begin{equation*}
        -\delta \e^{\frac{4(\eps-1)}{n-1}\psi}d(q,x)^{-2} \leq -b^{-2} \delta d(q,x)^{-2},
    \end{equation*}
    the condition $\Ric_\psi^{\infty}\geq -b^{-2} \delta d(q,x)^{-2}$ implies $\Ric_{\psi}^{\infty} \geq -\delta \e^{\frac{4(\eps-1)}{n-1}\psi(x)}d(q,x)^{-2}$. 
    Therefore, we can use the argument of the proof of Theorem \ref{lp-liouvolle-eps-range} under the assumptions of Theorem \ref{wu-lp-liouville} by taking $\delta$ sufficiently small. Hence, we obtain Theorem \ref{wu-lp-liouville}. \qed

\begin{remark}\label{delta-n-a}\rm{
    The original assertion of Theorem \ref{wu-lp-liouville} in \cite[Theorem 6.1]{jy-1} states that the $\delta$ depends only on $n$. However, from \eqref{beta-dependance}, we see that the $\delta$ depends on $n,A$.}
\end{remark}
\begin{remark}\label{remark-elliptic-harnack}\rm{
    In our proof of Theorem \ref{wu-lp-liouville}, while the first half of the proof to deduce \eqref{lp-lio-eq-3} is the same as the proof in \cite{jy-1},
    the second half, which is the argument of the second half of the proof of Theorem \ref{lp-liouvolle-eps-range} for the case $N = \infty$, is essentially different from \cite{jy-1}. 
    In particular, although a mean value inequality is deduced from the elliptic Harnack inequality and used to prove Theorem \ref{wu-lp-liouville} in \cite{jy-1}, our proofs do not go through the elliptic Harnack inequality.
    In addition, the alternative proof of Theorem \ref{lp-liouvolle-eps-range} is also different in that we used the maximum principle.}

\end{remark}
Although in Theorem \ref{lp-liouvolle-eps-range}, we make an assumption such as \eqref{odd-assumption}, this does not affect the $L^p$-Liouville theorems for $N\in [n,\infty]$, which are the main result of this section.

\section{Liouville type theorem for harmonic functions of sublinear growth under non-negative weighted Ricci curvature}
In this section, we prove the Liouville type theorem for harmonic functions of sublinear growth under non-negative weighted Ricci curvature with bounded weight function.
Although we considered the local curvature bounds $K_{\eps}(q,R)$, $\widetilde{K}_\infty(q,R)$ in previous sections, from now on, we only consider global curvature bounds and do not use the notation $K_{\eps}(q,R).$
We first show some functional inequalities under non-negative weighted Ricci curvature with bounded weight function.
\begin{lemma}(Neumann-Poincar\'{e} inequality)
    Let $(M,g,\mu)$ be an $n$-dimensional complete weighted Riemannian manifold and $N \in (-\infty,1] \cup [n,+\infty]$. We assume $\Ric_\psi^N \geq 0$ and $|\psi| \leq h$ for some positive constant $h$.
    Then, there exists a constant $\widetilde{C}_p$ depending on $h,n,N$ such that 
    \begin{equation}
    \forall f \in C^{\infty}(M), \quad \int_{B_q(R)}|f - f_{B_q(R)}|^2\  \d\mu \leq \widetilde{C}_p R^2\int_{B_q(2R)}|\nabla f|^2\  \d\mu
    \end{equation}
    for all $q\in M$ and $R > 0$.
\end{lemma}
\begin{proof}
    We choose a constant $\eps$ satisfying
    \begin{equation}\label{choose-epsilon}
        |\eps| < \min\left\{1,\sqrt{\frac{N-1}{N-n}}\right\}
    \end{equation}
    for $N\neq 1$ and $\eps = 0$ for $N = 1$.
    We set 
    \begin{equation*}
        a := \exp\left( -\frac{2(1-\eps)}{n-1}h \right),\quad b := \exp\left(\frac{2(1-\eps)}{n-1}h\right).
    \end{equation*}
    Since $-h \leq \psi \leq h$, we have 
    \begin{equation}\label{measure-pinching-2}
        a\leq \exp\left(\frac{2(1-\eps)}{n-1}\psi\right) \leq b.
    \end{equation}
    This pinching condition for the weight function and non-negativity of the $\Ric_\psi^N$ can be considered as the condition of the lower boundedness of $\Ric_\psi^N$ with $\eps$-range (note that the variable curvature bound degenerates to the constant curvature bound when we consider the non-negative curvature bound). 
    Hence, by Theorem \ref{poincare_thm}, we obtain the desired inequality.
\end{proof}

The local Sobolev type inequality, Poincar\'{e} type inequality and the modified local Sobolev type inequality are also obtained under the non-negativity of $\Ric_{\psi}^N$ and the boundedness of the weight function by regarding this assumption as a lower bound of $\Ric_\psi^N$ with $\eps$-range as above and applying the arguments in Section 3.
 From these inequalities, we derive the following mean value inequality.
\begin{theorem} (Mean  value inequality)\label{mean-ineq}
    Let $(M, g, \mu)$ be an $n$-dimensional complete weighted Riemannian manifold and $N \in (-\infty,1] \cup [n,+\infty]$. We assume $\Ric_\psi^N \geq 0$ and $|\psi | \leq h$ for some constant $h > 0$.
    Let $u$ be a smooth non-negative function satisfying $\Delta_\psi u \geq 0$.
    Given constants $0 < \theta < 1$ and $p > 0$, there exists a positive constant $C_{13}$ depending on $n,h,N, p, \theta$ such that 
    \begin{equation}
        \|u\|_{{\infty},\theta R} \leq C_{13} \left( \frac{1}{V_q(\theta R)}\int_{B_q(R)}u^p\  \d \mu \right)^{\frac{1}{p}}.
    \end{equation}
\end{theorem}
Using this mean value inequality, we show the following Liouville-type theorem for harmonic functions of sublinear growth. This is a generalization of Theorem \ref{wang-sublinear}.
We remark that $N\in[0,1]$ is excluded since we use the Bochner inequality.
\begin{theorem}\label{liouville-negative}
    Let $(M, g, \mu)$ be an $n$-dimensional complete weighted Riemannian manifold and $N \in (-\infty,0) \cup [n,+\infty]$.
    We assume $\Ric_\psi^N \geq 0$ and $|\psi|\leq h$ for some positive constant $h > 0$. 
    Let $u$ be a smooth $\psi$-harmonic function of sublinear growth.
    Then, $u$ is necessarily constant.
\end{theorem}
\begin{proof}
    We apply the argument in \cite[Theorem 3.2]{munteau_wang-1}.
Let $u$ be a $\psi$-harmonic function, which is of sublinear growth, i.e.,
\begin{equation}\label{sublinear-def}
    \lim _{d(q,x) \rightarrow \infty} \frac{|u(x)|}{d(q,x)}=0
\end{equation}
for some $q\in M$.
Combining $\Delta_\psi u = 0$ with the $N$-weighted Bochner inequality:
\begin{equation}\label{weighted_bochner_N}
    \Delta_{\psi}\left(\frac{|\nabla u|^2}{2}\right)-\left\langle\nabla \Delta_{\psi} u, \nabla u\right\rangle \geq \operatorname{Ric}_\psi^N(\nabla u)+\frac{\left(\Delta_{\psi} u\right)^2}{N},
\end{equation}
 we obtain
\begin{equation*}
    \Delta_\psi \left(\frac{|\nabla u|^2}{2}\right) \geq\Ric_\psi^N(\nabla u) \geq 0.
\end{equation*}
Therefore, we see that $|\nabla u|^2$ is $\psi$-subharmonic. 
Applying Theorem \ref{mean-ineq} to $|\nabla u|^2$, we obtain 
\begin{equation}\label{mean-val-ineq-sublinear}
    \sup _{B_q\left( R/2\right)}|\nabla u|^2 \leq \frac{C_{14}}{V_q\left(R/2\right)} \int_{B_q(R)}|\nabla u|^2 \ \d \mu
\end{equation}
for some constant $C_{14}$ depending on $n,h,N$. 

We estimate the right-hand side of \eqref{mean-val-ineq-sublinear}.
We set a cut-off function $\phi$ such that $\phi=1$ on $B_q(R), \phi=0$ on $M \setminus B_q(2 R)$ and $|\nabla \phi| \leq \frac{2}{R}$.
From the definition of $\phi$, we get
\begin{eqnarray}
    \int_M|\nabla u|^2 \phi^2\  \d \mu & =& -2 \int_M u \phi\langle\nabla u, \nabla \phi\rangle \ \d \mu\nonumber \\
& \leq& 2 \int_M|u| | \phi| |\langle\nabla u, \nabla \phi\rangle| \  \d \mu\nonumber \\
& \leq& \frac{1}{2} \int_M|\nabla u|^2 \phi^2 \ \d \mu+2 \int_M u^2|\nabla \phi|^2\  \d \mu\label{main-thm-eq-1} ,
\end{eqnarray}
where we used the integration by parts and $\Delta_\psi u = 0$ in the first inequality.
Combining this with the volume growth estimate as in Theorem \ref{bishop_gromov}, we have
\begin{eqnarray}
\int_{B_q(R)}|\nabla u|^2 \ \d \mu & \leq& 4 \int_M u^2|\nabla \phi|^2 \ \d \mu\nonumber \\
&\leq& \frac{16}{R^2} \int_{B_q(2 R) \setminus B_q(R)} u^2 \ \d \mu\nonumber \\
& \leq& \frac{16}{R^2}\left(\sup _{B_q(2 R)} u^2\right) V_q(2 R) \label{main-thm-eq-3}\\
& \leq& \frac{C_{15}}{R^2} \left(\sup _{B_q(2 R)} u^2\right) V_q(R/2)\label{main-thm-eq-2},
\end{eqnarray}
where $C_{15}$ is a constant depending on $n,h,N$.  
From \eqref{main-thm-eq-2}, we can estimate the right-hand side of \eqref{mean-val-ineq-sublinear} as
\begin{eqnarray}
    \frac{1}{V_q(R/2)}\int_{B_q(R)}|\nabla u|^2 \ \d \mu
    &\leq& \frac{C_{15}}{R^2} \left(\sup _{B_q(2 R)} u^2\right)\label{final-eq}.
\end{eqnarray}
Since $u$ is of sublinear growth, we see that the right-hand side of \eqref{final-eq} goes to $0$ as $R\rightarrow \infty$. Combining this with \eqref{mean-val-ineq-sublinear}, we have $|\nabla u| = 0$ and the theorem follows.
\end{proof} 
The fact that the variable curvature bound with $\eps$-range \eqref{eps-range-intro} becomes the non-negative curvature bound when $K = 0$ enables us to use the functional inequalities with $\eps$-range obtained in Section 3.
Hence, we could apply the argument in \cite{munteau_wang-1} to obtain Theorem \ref{liouville-negative}.

\section{Gradient estimate}
In this section, we prove the following gradient estimate of harmonic functions under lower bounds of $\Ric_\psi^N$ with $\eps$-range.
This is an alternative version of Theorem \ref{grad-est-thm}.
\begin{theorem}\label{grad-est-eps-range}
    Let $(M,g,\mu)$ be an $n$-dimensional complete weighted Riemannian manifold and $N \in (-\infty,0) \cup [n,\infty]$,
    $\ez \in \R$ in the $\ez$-range \eqref{epsilin-range}, $K \geq 0$ and $b \ge a>0$.
    Assume that
    \[ \Ric_\psi^N(v)\ge -K\e^{\frac{4(\ez-1)}{n-1}\psi(x)} g(v,v) \]
    holds for all $v \in T_xM \setminus 0$ and
    \begin{equation}\label{measure-pinching-3}
    a \le \e^{-\frac{2(\ez-1)}{n-1}\psi} \le b.
    \end{equation}
    Let $u$ be a smooth positive function satisfying $\Delta_\psi u = 0$.
    Then there exists a constant $C_{16}$ depending on $c,a,b,n,K$ such that
   $$
   |\nabla \log u| \leq C_{16}.
   $$
\end{theorem}
\begin{proof}
    We apply the argument in \cite[Theorem 3.1]{soliton}.
 From the $N$-weighted Bochner inequality \eqref{weighted_bochner_N}, we deduce
 \begin{equation}\label{grad-est-eq-1}
    \Delta_{\psi}\left(\frac{|\nabla u|^2}{2}\right)-\left\langle\nabla \Delta_{\psi} u, \nabla u\right\rangle  \geq \operatorname{Ric}_\psi^N(\nabla u)+\frac{\left(\Delta_{\psi} u\right)^2}{N} = \Ric_\psi^N(\nabla u) \geq -\frac{K}{a^2}|\nabla u|^2.
\end{equation}
Applying Theorem \ref{mean-eq-thm}, there exists a constant $C_{17}$ depending on $c,a,b,n,K$ such that the following mean value inequality holds:
\begin{equation}\label{grad-est-eq-3}
    \sup _{B_x\left(1/16\right)}|\nabla u|^2 \leq \frac{C_{17}}{V_x\left(1/16\right)} \int_{B_x\left(1/8\right)}|\nabla u|^2 \ \d\mu
\end{equation}
for arbitrary $x \in M$. We estimate the right-hand side of \eqref{grad-est-eq-3}.
Now, let $\phi$ be a cut-off function with support in $B_x\left(\frac{1}{4}\right)$ such that $\phi=1$ on $B_x\left(\frac{1}{8}\right)$ and $|\nabla \phi| \leq 16$. 
Applying the same argument as the one for showing \eqref{main-thm-eq-3} for $R = \frac{1}{8}$, we have 
\begin{equation*}
     \int_{B_x\left(1/8\right)}|\nabla u|^2 \ \d\mu \leq 16\cdot 8^2 V_x(1/4)\left(\sup _{B_x\left(1/4\right)} u\right)^2.
\end{equation*}
Combining this with \eqref{grad-est-eq-3} and the volume growth estimate in Theorem \ref{bishop_gromov}, we get
\begin{equation}\label{grad-est-eq-2}
    |\nabla u|(x) \leq C_{18} \sup _{B_x\left(1/4\right)} u ,
\end{equation}
where $C_{18}$ is a constant depending on $c,a,b,n,K$.

On the other hand, we have the following mean value inequality:
\begin{equation}\label{mean-val-eq-for-sup-2}
    \left( \frac{1}{V_x(1/2)}\int_{B_x(1/2)} u^k \ \d\mu\right)^{\frac{1}{k}}  \leq \widetilde{C}_{sup} \inf_{B_x(1/4)}u
\end{equation}
for sufficiently small $k > 0$.
This inequality is proved by applying the argument in \cite[Lemma 11.2]{p_li} (Theorem \ref{sup-lemma}).
We remark that, although the argument to prove Theorem \ref{sup-lemma} requires the lower Neumann-Poincar\'{e} eigenvalue bound \eqref{p-li-neu-poincare}, by using the Neumann-Poincar\'{e} inequality in Theorem \ref{poincare_thm} instead, we can follow the argument in \cite[Lemma 11.2]{p_li} to prove \eqref{mean-val-eq-for-sup-2}.
From Theorem \ref{mean-eq-thm}, we also have the following mean value inequality:
\begin{equation*}
    \|u\|_{\infty,\frac{1}{4}}\leq \widetilde{C}_{sub}\left( \frac{1}{V_x(1/4)}\int_{B_x(1/2)} u^k \ \d\mu\right)^{\frac{1}{k}}.
\end{equation*}
Combining this with \eqref{mean-val-eq-for-sup-2}, we obtain the following elliptic Harnack inequality:
$$
\sup _{B_x\left(1/4\right)} u \leq \widetilde{C}_{H} \inf _{B_x\left(1/4\right)} u,
$$
where $\widetilde{C}_{H}$ is a constant depending only on $c,a,b,n,K$. Combining this with \eqref{grad-est-eq-2}, we obtain
$$
|\nabla u|(x) \leq C_{18}\widetilde{C}_{H} u(x).
$$
Therefore, the theorem follows.
\end{proof}
We note that the following simple corollary holds.
\begin{corollary}
    Let $(M,g,\mu)$ be an $n$-dimensional complete weighted Riemannian manifold and $N\in (-\infty,0)\cup [n,\infty]$. We assume that $\Ric_{\psi}^N\geq 0$ and $|\psi| \leq h$ for some positive constant $h$.
    Let $u$ be a smooth positive function satisfying $\Delta_\psi u = 0$. Then, there exists $C_{19}$ depending on $n,h,N$ such that 
    \begin{equation*}
        |\nabla\log u|\leq C_{19}.
    \end{equation*}
\end{corollary}

In \cite{soliton}, the volume growth estimate is obtained without imposing the boundedness of the weight function and this leads to the gradient estimate of $\psi$-harmonic functions in Theorem \ref{grad-est-thm}. 
In this section, we assume that the weight function is bounded and then, the volume growth estimate and functional inequalities with $\eps$-range are used to derive the gradient estimate of $\psi$-harmonic functions. Hence, our gradient estimate is not a direct generalization of Theorem \ref{grad-est-thm}.

\textbf{Acknowledgement}

I would like to express my deepest gratitude to my supervisor Shin-ichi Ohta, who gave me many valuable suggestions on preliminary versions of this paper.

\end{document}